\documentclass[11pt]{article}

\usepackage{fullpage}
\usepackage{amsmath}
\usepackage{amstext}
\usepackage{amssymb}
\usepackage{amsthm}
\usepackage{amsfonts}
\usepackage{amsxtra}
\usepackage{mathrsfs}
\usepackage{mathtools}
\usepackage{dsfont}
\usepackage{breqn}
\usepackage{enumerate}
\usepackage{bbm}
\usepackage{comment}
\usepackage[utf8]{inputenc}

\usepackage{sty}

\usepackage{tikz}
\usetikzlibrary{patterns}
\usetikzlibrary{hobby} 
\usetikzlibrary{arrows.meta} 
\tikzset{>={Latex[length=4,width=4]}} 
\usetikzlibrary{calc,intersections,decorations.markings}
\usepackage{siunitx}

\colorlet{mylightblue}{blue!20}
\colorlet{myblue}{blue!50!black}
\colorlet{mydarkblue}{blue!30!black}
\colorlet{mylightred}{red!10}
\colorlet{myred}{red!50!black}
\colorlet{mydarkred}{red!60!black}
\colorlet{mydarkgreen}{green!30!black}

\tikzset{
midarr/.style={decoration={markings,mark=at position #1 with {\arrow{stealth}}},postaction={decorate}},
midarr/.default=0.5
}

\usepackage{wrapfig}

\newtheorem{theorem}{Theorem}[section]
\newtheorem{lemma}[theorem]{Lemma}

\newtheorem{remark}[theorem]{Remark}

\newtheorem{definition}[theorem]{Definition}

\usepackage{hyperref}
\usepackage{xcolor}
\definecolor{OliveGreen}{HTML}{3C8031}
\hypersetup{
colorlinks=true,
urlcolor=blue,
linkcolor=blue,
citecolor=OliveGreen,
}

\title{Information-Theoretic Thresholds for Planted Dense Cycles}
\usepackage{authblk}
\author[1]{Cheng Mao\thanks{Email: \textit{cheng.mao@math.gatech.edu}. C.M.\ was supported in part by NSF grants DMS-2053333 and DMS-2210734.}}
\author[2]{Alexander S.\ Wein\thanks{Email: \textit{aswein@ucdavis.edu}. A.S.W.\ was supported in part by an Alfred P.\ Sloan Research Fellowship.}}
\author[1]{Shenduo Zhang\thanks{Email: \textit{szhang705@gatech.edu}. S.Z.\ was supported in part by NSF grant DMS-2053333.}}

\affil[1]{School of Mathematics, Georgia Institute of Technology}

\affil[2]{Department of Mathematics, University of California, Davis}
\date{\today}

\begin{document}

\maketitle

\begin{abstract}
We study a random graph model for small-world networks which are ubiquitous in social and biological sciences. In this model, a dense cycle of expected bandwidth $n \tau$, representing the hidden one-dimensional geometry of vertices, is planted in an ambient random graph on $n$ vertices. For both detection and recovery of the planted dense cycle, we characterize the information-theoretic thresholds in terms of $n$, $\tau$, and an edge-wise signal-to-noise ratio $\lambda$. In particular, the information-theoretic thresholds differ from the computational thresholds established in a recent work for low-degree polynomial algorithms, thereby justifying the existence of statistical-to-computational gaps for this problem.
\end{abstract}

\newpage

\tableofcontents

\newpage

\section{Introduction}

The \emph{Watts--Strogatz small-world model} has been an influential random graph model since its proposal in 1998 due to the ubiquity of the small-world phenomenon in complex networks \cite{watts1998collective,watts2004small}.
In this model, there are $n$ vertices with latent positions on a circle, and the vertices are more likely to be connected to their $k$-nearest geometric neighbors than to more distant vertices.
In other words, a denser cycle of length $n$ and width $k$ is ``planted" in the sparser ambient random graph on $n$ vertices. 
Informally, the small-world model can also be viewed as an interpolation between a \emph{random geometric graph} \cite{penrose2003random}, where edges exist only between vertices with nearby locations on a circle, and an \emph{\ER\ graph} \cite{erdos1959random}, where edges are random and independent.
As a consequence, a small-world network tends to have a high clustering coefficient due to the geometry while preserving low distances between vertices in a random graph.

While there has been extensive literature on small-world networks and geometric graphs, the associated statistical problems, such as detection and recovery of the latent geometry from the observed random graph, have only gained attention more recently.
The information-theoretic thresholds and efficient algorithms for the small-world model are studied in \cite{cai2017detection}, but there remain several gaps between upper and lower bounds that are unknown to be inherent or not.
Much sharper characterizations of the recovery thresholds are given in \cite{bagaria2020hidden,ding2020consistent} but only for a small bandwidth parameter $k = n^{o(1)}$. 
From the perspective of graphon estimation, there have also been algorithms and statistical analyses introduced for related models recently \cite{janssen2022reconstruction,natik2021consistency,giraud2021localization}. 
Moreover, the study of random geometric graphs in higher dimensions has received broad interests in recent years \cite{bubeck2016testing,araya2019latent,brennan2020phase,eldan2022community,liu2022testing,li2023spectral}; see the survey \cite{duchemin2022random} and, in particular, the recent works \cite{liu2021phase,liu2021probabilistic} for high-dimensional random geometric graphs with edge noise similar to what we consider.

Since the model of interest consists of a hidden dense cycle planted in a sparser random graph, the recent work \cite{maoDetectionRecoveryGapPlanted2023} refers to it as the \emph{planted dense cycle} problem, following the etymology of planted clique and planted dense subgraph problems \cite{jerrum1992large,kuvcera1995expected,ames2015guaranteed,hajek2015computational}.
The work \cite{maoDetectionRecoveryGapPlanted2023} studies the problem in the framework of \emph{low-degree polynomial algorithms}, a framework that has proved to be successful at predicting \emph{computational} thresholds, i.e., understanding when computationally efficient (polynomial time) algorithms can solve a statistical task~\cite{hopkins2017efficient,sos-hidden,sam-thesis,kunisky2019notes,schramm2020computational}.
However, the more fundamental \emph{information-theoretic} (or \emph{statistical}) thresholds---where no constraints are placed on computation time---remain open, and we aim to address them in this work. 
More specifically, suppose that in an ambient random graph on $n$ vertices with edge density $q$, there is a hidden cycle with expected width $k = n \tau$ and edge density $p$. 
We find the information-theoretic thresholds for detecting the presence of the cycle and for recovering the location of the cycle, in terms of the parameters $n,\tau,p,q$. 
In particular, the information-theoretic thresholds for detection and recovery both differ from the computational thresholds given in \cite{maoDetectionRecoveryGapPlanted2023}, justifying the existence of \emph{statistical-to-computational gaps} for this problem.





\subsection{Problem setup}
The planted dense cycle model can be described as follows.
For any $a,b \in [0,1]$, define
$$
\dist(a,b) := \min\{|a-b|, 1-|a-b|\} .
$$
In other words, $\dist(a,b)$ is the distance between $a$ and $b$ on a circle of circumference $1$.
Throughout the paper, we consider the setting where the number of vertices $n$ grows, and other parameters $p, q, r, \tau \in [0,1]$ may depend on $n$.
In the models to be defined, $p$ and $q$ will be the average edge densities on and off the planted cycle respectively, $r$ is the average edge density of the entire graph, and $n \tau$ is the bandwidth of the cycle, satisfying
\begin{equation}
n \to \infty, \qquad 
0 < q < r < p \le 1, \qquad
0 < \tau < 1/2, \qquad
r = \tau p + (1-\tau) q .
\label{eq:parameters-all}
\end{equation}
With these parameters, the models are formally defined as follows.
Let $[n] := \{1, \dots, n\}$ and $\binom{[n]}{2} := \{(i,j) : i, j \in [n], \, i < j\}$.

\begin{definition}[Model $\cP$, planted dense cycle] \label{def:mod-p}
Let $z \in [0,1]^n$ be a latent random vector whose entries $z_1, \dots, z_n$ are i.i.d.\ $\Unif([0,1])$ variables.
Let 
\begin{equation*}
X_{ij} := \1_{\dist(z_i,z_j) \le \tau/2} \quad \text{for all } (i,j) \in \binom{[n]}{2} .
\end{equation*}
For concreteness, we let $X_{ii} = 0$ and $X_{ji} = X_{ij}$ for $i \ne j$, so that $X \in \R^{n \times n}$ is the adjacency matrix of the underlying cycle.
We observe an undirected graph with adjacency matrix $A \in \R^{n \times n}$ whose edges, conditional on $z_1, \dots, z_n$, are independently sampled as follows:
$A_{ij} \sim \Bern(p)$ if $X_{ij} = 1$ and $A_{ij} \sim \Bern(q)$ if $X_{ij} = 0$ for $(i,j) \in \binom{[n]}{2}$.
We write $A \sim \cP_A$ and $(A,X) \sim \cP$ (or, equivalently, $(A,z) \sim \cP$).
\end{definition}


\begin{definition}[Model $\cQ$, \ER graph] \label{def:mod-q}
We observe a $G(n,r)$ \ER graph with adjacency matrix $A \in \R^{n \times n}$.
We write $A \sim \cQ$.
\end{definition}

We now formulate the detection and recovery problems of interest.

\begin{definition}[Detection]
\label{def:prob-detect}
Let $\cP$ and $\cQ$ be the models from Definitions~\ref{def:mod-p} and~\ref{def:mod-q} respectively, with parameters in \eqref{eq:parameters-all}.
Observing the adjacency matrix $A \in \R^{n \times n}$ of a graph, we test $H_1 : A \sim \cP_A$ against $H_0 : A \sim \cQ$.
We say that a test $\Phi$, a $\{0,1\}$-valued measurable function of the observation $A$, achieves 
\begin{itemize}
\item strong detection, 
if
$\lim_{n\to \infty} [\cP\br{\Phi(A) = 0} + \cQ\br{\Phi(A) = 1}] = 0;$
\item weak detection, if 
$\limsup_{n\to \infty} [\cP\br{\Phi(A) = 0} + \cQ\br{\Phi(A) = 1}] < 1.$
\end{itemize}
\end{definition}

\begin{definition}[Recovery]
\label{def:prob-recover}
Let $\cP$ be the model
from Definition~\ref{def:mod-p} with parameters in \eqref{eq:parameters-all}.
For $(A,X) \sim \cP$, observing $A$, we aim to estimate $X$ with an estimator $\hat X \in \R^{n \times n}$ that is measurable with respect to $A$.
Consider the mean squared error $R(\hat X, X) := \sum_{1 \le i < j \le n} \E[(\hat X_{ij} - X_{ij})^2]$, where the expectation is with respect to $(A,X) \sim \cP$.
We say that an estimator $\hat X$ achieves
\begin{itemize}
\item strong recovery, if $\lim_{n \to \infty} \frac{R(\hat X, X)}{\binom{n}{2} \tau (1-\tau)} = 0$;
\item weak recovery, if $\limsup_{n \to \infty} \frac{R(\hat X, X)}{\binom{n}{2} \tau (1-\tau)} < 1$.
\end{itemize}
\end{definition}

\noindent
Note that each $X_{ij}$ is marginally a $\Bern(\tau)$ random variable, so estimating $X_{ij}$ by its mean $\tau$ for all $(i,j) \in \binom{[n]}{2}$ yields a trivial mean squared error $\binom{n}{2} \tau (1-\tau)$, which justifies the above definition.

\subsection{Main results}
Our main results are summarized in the following theorem.

\begin{theorem}[Information-theoretic thresholds]\label{thm:main result}
Consider the detection and recovery problems in Definitions~\ref{def:prob-detect} and~\ref{def:prob-recover} respectively, with parameters $n, \tau, p, q, r$ in \eqref{eq:parameters-all}.
Furthermore, suppose that $(\log n)^3 \le n \tau \le \frac{n}{(\log n)^2}$ and $\frac{\log n}{n} \le r \le \frac 12$. 
Define $\lambda:=\frac{(p-q)^2}{r(1-r)}$ which can be seen as the signal-to-noise ratio of the problem.
Then we have:
\begin{itemize}
\item 
If $n \tau \lambda \to 0$ as $n \to \infty$, then no test achieves weak detection, and no estimator achieves weak recovery.

\item
If $\frac{n \tau \lambda}{\log n} \to \infty$ and $\frac{n \tau (p-r)}{\log n} \to \infty$ as $n \to \infty$, then there is a test that achieves strong recovery, and there is an estimator that achieves strong recovery.
\end{itemize}
\end{theorem}

\noindent
The four results, upper and lower bounds for detection and recovery, are established in Theorems~\ref{thm:detection lower bound}, \ref{thm:recovery lower bound}, \ref{thm:upper bound on detection}, and~\ref{thm:recovery-upper-bound-thm} respectively.
For simplicity, we have focused on a regime that is common to all the four theorems:
(1) the condition $(\log n)^3 \le n \tau \le \frac{n}{(\log n)^2}$ allows the average bandwidth $n \tau$ of the planted dense cycle to vary in a large range, but excludes extreme cases;
(2) the condition $\frac{\log n}{n} \le r \le \frac 12$ on the average edge density $r$ allows the random graph to be sparse (above the connectivity threshold) or dense.

Our conditions for positive and negative results match up to a logarithmic factor in most cases. 
The threshold for both detection and recovery is at $n \tau \lambda = \tilde \Theta(1)$, where, again, $n \tau$ is the average bandwidth of the planted dense cycle, and $\lambda$ is the edge-wise signal-to-noise ratio.

There is only one small regime where the gap between our positive and negative conditions is larger than a logarithmic factor. 
Namely, a gap appears when $n \tau \lambda = n \tau (p-r) \frac{p-r}{r (1-r) (1-\tau)^2} \gg 1$ and $n \tau (p-r) \ll 1$, in which case we have $p \gg r$ and $n \tau p \ll 1$. In this pathological regime, the expected number of edges on the planted dense ``cycle" is of order $n^2 \tau p \ll n$, smaller than the necessary number of edges to form an actual (even width-$1$) cycle through all the $n$ vertices.
However, the nonexistence of a cycle does not preclude the possibility of consistent detection nor recovery.
We choose not to pursue this corner case.

\begin{proof}
[Proof of Theorem~\ref{thm:main result}]
It suffices to verify the assumptions of Theorems~\ref{thm:detection lower bound}, \ref{thm:recovery lower bound}, \ref{thm:upper bound on detection}, and~\ref{thm:recovery-upper-bound-thm}.  
The assumptions of Theorem~\ref{thm:detection lower bound} are obvious.
For Theorem~\ref{thm:recovery lower bound}, the only less obvious condition is $n \lambda \tau^{3/2} \log n \to 0$, which follows from $n \tau \lambda \to 0$ together with $\tau \le \frac{1}{(\log n)^2}$.
For Theorem~\ref{thm:upper bound on detection}, note that $p-r = (1-\tau)(p-q) \approx p-q$ since $r = \tau p + (1-\tau) q$ and $\tau \to 0$. Thus the two conditions in Theorem~\ref{thm:upper bound on detection} are weaker than $\frac{n \tau \lambda}{\log n} \to \infty$ and $\frac{n \tau (p-r)}{\log n} \to \infty$.
Finally, for Theorem~\ref{thm:recovery-upper-bound-thm}, 
we already have the second condition $\frac{n \tau (p-q)}{\log n} \to \infty$. 
For the first condition $\frac{n \tau (p-q)^2}{p \log n} \to \infty$ to fail, we must $p/r \to \infty$ because $\frac{n \tau \lambda}{\log n} \approx \frac{n \tau (p-q)^2}{r \log n} \to \infty$. However, we then obtain $\frac{n \tau (p-q)^2}{p \log n} \approx \frac{n \tau (p-q)}{\log n} \to \infty$.
\end{proof}

Our information-theoretic results for the detection and recovery of a planted dense cycle complement the work \cite{maoDetectionRecoveryGapPlanted2023} which studies computationally efficient algorithms for the same model.
To be more precise, the work \cite{maoDetectionRecoveryGapPlanted2023} focuses on a class of algorithms based on low-degree polynomials in the entries of the adjacency matrix $A$, and the results for these low-degree algorithms can be summarized as follows.


\begin{remark}[Informal summary of computational thresholds \cite{maoDetectionRecoveryGapPlanted2023}]
\label{rmk:computation-results}
Consider the detection and recovery problems in Definitions~\ref{def:prob-detect} and~\ref{def:prob-recover} respectively,
with parameters in \eqref{eq:parameters-all} 
satisfying further $\constpq q \leq p \leq \constpq' q$ for some constants $\constpq' > \constpq > 1$. 
Then the detection threshold for the class of low-degree polynomial algorithms is at $n^3 p^3 \tau^4 = n^{o(1)}$, and the recovery threshold for the class of low-degree polynomial algorithms is at $n p \tau^2 = n^{o(1)}$.
%
\end{remark}

In the regime where $\constpq q \leq p \leq \constpq' q$, our signal-to-noise ratio $\lambda = \frac{(p-q)^2}{r(1-r)}$ has the same order as $p$, so the information-theoretic threshold from Theorem~\ref{thm:main result} can be expressed as $n p \tau = \tilde \Theta(1)$.
Therefore, Theorem~\ref{thm:main result} and Remark~\ref{rmk:computation-results} together suggest that there are statistical-to-computational gaps for both detection and recovery of a planted dense cycle.
To further illustrate the gaps, we draw a phase diagram of the thresholds in Figure~\ref{fig:phase}, which has the additional information-theoretic threshold compared to Figure~1 of \cite{maoDetectionRecoveryGapPlanted2023}. 
Let $p = n^{-a}$ and $\tau = n^{-b}$ for constants $a, b \in (0,1)$.
The information-theoretic threshold for both detection and recovery is given by $1-a-b = 0$ (black solid line). The computational threshold for detection is given by $3 - 3a - 4b = 0$ (red dotted line), while the computational threshold for recovery is given by $1- a - 2b = 0$ (blue dashed line). 
In particular, in region $B$ of the figure, detection is information-theoretically possible, but computationally hard (for low-degree algorithms); in region $B \cup C$, recovery is information-theoretically possible but computationally hard. 

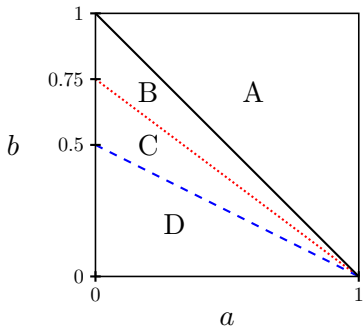
\begin{figure}[ht]
\centering

\begin{tikzpicture}[scale=0.35]
\def\xtick#1#2{\draw[thick] (#1)++(0,.2) --++ (0,-.4) node[below=-.5pt,scale=0.7] {#2};}
\def\ytick#1#2{\draw[thick] (#1)++(.2,0) --++ (-.4,0) node[left=-.5pt,scale=0.7] {#2};}

\coordinate (O) at (0,0);
\coordinate (NE) at (10,10);
\coordinate (NW) at (0,10);
\coordinate (SE) at (10,0);
\coordinate (W1) at (0,5);
\coordinate (W2) at (0,7.5);
\coordinate (S1) at (5,0);

\def\DE{(W2) to (SE)}
\def\RE{(W1) to (SE)}
\path[name path=DE] \DE;
\path[name path=RE] \RE;
\path[name path=AB] (NW) -- (SE);



\node at (6,7) {A};
\node at (2,5) {C};
\node at (3,2) {D};
\node at (2,7) {B};

\draw[red,thick,densely dotted] \DE;
\draw[blue,thick,dashed] \RE;
\draw[black,thick] (NW) -- (SE);




\draw[semithick] (O) rectangle (NE);
\node[left=25pt] at (W1) {$b$};
\node[below=10pt] at (S1) {$a$};
\xtick{O}{0}
\xtick{SE}{1}
\ytick{O}{0}
\ytick{NW}{1}
\ytick{W1}{0.5}
\ytick{W2}{0.75}
\end{tikzpicture}

\caption{The statistical-to-computational gaps for detection and recovery of a planted dense cycle with $p = n^{-a}$, $\tau = n^{-b}$, and $\constpq q \leq p \leq \constpq' q$. Both tasks become more difficult as $a$ and $b$ increase. Black solid line: information-theoretic threshold for both detection and recovery. Red dotted line: computational (low-degree) threshold for detection. Blue dashed line: computational threshold for recovery.}
\label{fig:phase}
\end{figure}

\subsection{Technical contributions}
The positive results, i.e., upper bounds, for detection and recovery are achieved by maximizing $\langle X, A \rangle$ over the set of realizable cycles $X$ in model $\cP$ from Definition~\ref{def:mod-p}.
It is fairly straightforward to analyze such an exhaustive search; see Section~\ref{sec:upper-bounds}.
Our main technical contributions lie in the proofs of negative results, especially the detection lower bound, which we now explain.

Proving a lower bound for detection is equivalent to bounding the total variation distance between the models $\cP$ and $\cQ$. 
Since $\cP$ is a Bayesian model with a prior on $z$ (or, equivalently, on $X$), the natural approach is the second moment method which controls the total variation distance by bounding the $\chi^2$-divergence between $\cP$ and $\cQ$. 
However, as we will see in Section~\ref{sec:lower-bound-detect}, this is not sufficient for obtaining the desired detection lower bound, because the $\chi^2$-divergence blows up due to a certain rare event.
Instead, we need to condition on the complement of such a rare event, resulting in a \emph{conditional second moment method}. See, e.g., \cite{wuTestingCorrelationUnlabeled2021,liu2022testing} for other applications of the conditional second moment method to proving lower bounds for statistical testing problems on random graphs.

Furthermore, the analysis of the conditional second moment is highly nontrivial. 
Recall that the underlying cycle $X$ is defined by $X_{ij} = \bbone_{\dist(z_i,z_j) \le \tau/2}$ which involves 
two random variables $z_i$ and $z_j$ at once. 
This leads to intricate dependencies between entries of $X$ which in turn pose challenges to proving concentration inequalities for functions of $X$.
With typical methods such as the bounded difference inequality and the martingale method (see, e.g., \cite{warnkeMethodTypicalBounded2016}) failing to yield the desired bound, we resort to decoupling and concentration inequalities for \emph{$U$-statistics} \cite{de1995decoupling,delapenaDecoupling1999,gineExponentialMomentInequalities2000,houdreExponentialInequalitiesConstants2003,10.1214/009117906000000476}. 
However, the off-the-shelf results fall short of giving the optimal condition; see Remark~\ref{rmk:technical}.
To overcome this difficulty, we carefully split the entries of $X$ into independent groups (cf.\ Lemma~\ref{lem:independence of U}) and then utilize the independence across groups. 
This strategy is inspired by the analysis of graph-dependent random variables in the seminal work \cite{janson2004large}. 
See Section~\ref{sec:concentration-u-statistics-three-steps} for the proof and further discussion.

Finally, the proof of the negative result for the recovery problem follows an existing approach which relates the \emph{minimum mean squared error} (MMSE) to the mutual information between the observed adjacency matrix $A$ and the hidden cycle $X$. The connection of the MMSE to derivatives of entropy-related quantities has long been used in the statistics literature (cf.\ Chapter~4 of the textbook \cite{lehmann2006theory}). 
The precise relation between the MMSE and mutual information is given by the I-MMSE formula in the Gaussian setting \cite{guo2005mutual} and also by an area theorem in a more general setting \cite{measson2004life,measson2009generalized}. This approach finds a variety of applications such as proving negative results for the stochastic block model \cite{deshpandeAsymptoticMutualInformation2016} and graph matching \cite{wuSettlingSharpReconstruction2022}. 
We use the same strategy but the adaptation to our model is nontrivial. See Section~\ref{sec:lower-bound-recovery} for a complete proof.

\section{Lower bound for detection}
\label{sec:lower-bound-detect}

The goal of this section is to prove the lower bound for detection as stated in Theorem~\ref{thm:detection lower bound} below. Let $\TV(\cP,\cQ)$ denote the total variation distance between probability distributions $\cP$ and $\cQ$.
By the Neyman--Pearson lemma, for any fixed $n$, we have
$$
\inf_{\Phi} [\cP\br{\Phi(A) = 0} + \cQ\br{\Phi(A) = 1}] = 1 - \TV(\cP, \cQ) ,
$$
where the infimum is over all tests $\Phi$ measurable with respect to the observation $A \in \R^{n \times n}$.
As a result, if there is a test $\Phi$ achieving weak detection in the sense of Definition~\ref{def:prob-detect}, then $\TV(\cP, \cQ)$ must be bounded away from $0$ as $n \to \infty$.
Conversely, if we could show that $\TV(\cP, \cQ) \to 0$ as $n \to \infty$, the impossibility of weak detection would follow, thereby proving the detection lower bound in Theorem~\ref{thm:main result}.

\begin{theorem}
\label{thm:detection lower bound}
Consider the models $\cP$ and $\cQ$ from Definitions~\ref{def:mod-p} and~\ref{def:mod-q} respectively with parameters $n, \tau, p, q, r$. 
Define $\lambda:=\frac{(p-q)^2}{r(1-r)}.$
Suppose that $\tau \to 0$, $n \tau \ge (\log n)^3$, and 
\begin{equation}\label{eq:cond_detection_lower_bound}
n\tau \lambda \to 0
\end{equation}
as $n \to \infty$. Then we have $\TV\pr{\cP_A,\cQ} \to 0$.
\end{theorem}


\subsection{Conditional second moment method}

Let $\chi^2(\cP,\cQ)$ denote the $\chi^2$-divergence between $\cP$ and $\cQ$. It is well-known \cite{tsybakovIntroductionNonparametricEstimation2009} that
\begin{equation*}
\TV(\cP,\cQ) \leq \chi^2(\cP,\cQ) .
\end{equation*}
To prove Theorem~\ref{thm:detection lower bound}, it is tempting to show that $\chi^2(\cP,\cQ) \to 0$.
However, this only holds under an assumption stronger than desired (cf.\ Remark~\ref{rmk:technical}).
The $\chi^2$-divergence between $\cP$ and $\cQ$ can be large due to some rare event, and one strategy to bypass this issue is to condition on the complement of such a rare event, known as the conditional second moment method. We now formulate this strategy as a lemma, following the presentation in \cite{wuTestingCorrelationUnlabeled2021}. 

Recall that $\cP$ defines a joint distribution of $(A,z)$ in Definition~\ref{def:mod-p}. 
For simplicity, we use $\cP(A,z)$, $\cP(A \mid z)$, and $\cP(A)$ to denote the joint, conditional, and marginal densities respectively.

\begin{lemma} \label{lem:conditional-second-moment}
Consider the distributions $\cP$ and $\cQ$ from Definitions~\ref{def:mod-p} and~\ref{def:mod-q} respectively.
Recall the latent vector $z \in [0,1]^n$ from Definition~\ref{def:mod-p} whose entries are i.i.d.\ uniform in $[0,1]$. Let $z'$ be an independent copy of $z$.
Let $\cE$ denote an event measurable with respect to $z$ (see \eqref{eq:goodEvent} for the specific choice of $\cE$). 
If $\cP(\cE) = 1-o(1)$ and 
\begin{equation}
\label{eq:goal-of-conditional-second-moment}
\E_{z,z'} \E_{A \sim \cQ} \kr{ \frac{ \cP(A \mid z) \, \cP(A \mid z') }{ \cQ(A)^2 } \bbone_{z \in \cE, \, z' \in \cE} } = 1+o(1),
\end{equation}
then we have $\TV\pr{\cP_A,\cQ} = o(1)$.
\end{lemma}

\begin{proof}
We define $\cP'$ to be the planted model conditional on $\goodEvent$, i.e.,
\begin{equation}\label{eq:def planted model}
\cP'(A,z) := \frac{\cP(A,z)}{\cP(\goodEvent)} \1_{\zb\in\goodEvent}=(1+o(1)) \cP(A,z) \1_{\zb\in\goodEvent} .
\end{equation}
Together with $\cP(\cE) = 1-o(1)$, this implies
\begin{align*}
\TV(\cP_A, \cP'_A) &= \sum_A |\cP(A) - \cP'(A)| 
\le \sum_A \int |\cP(A,z) - \cP'(A,z)| \, dz \\
&\le \sum_A \bigg( \int_{\cE} o(1) \cdot \cP(A,z) \, dz + \int_{\cE^c} \cP(A,z) \, dz \bigg)
= o(1) + \cP(\cE^c) = o(1) .
\end{align*}

Moreover, the likelihood ratio between $\cP'$ and $\cQ$
is given by
\begin{equation*}
\frac{\cP'}{\cQ} (A) = \frac{\int \cP'(A,z) \, dz}{\cQ(A)} 
= \pr{1+o(1)}\frac{\int_{\cE} \cP(A,z) \, dz}{ \cQ(A)}  .
\end{equation*}
As a result,
\begin{equation*}
\E_{A \sim \cQ} \kr{ \pr{\frac{ \cP'}{ \cQ} (A)}^2 } 
= (1+o(1)) \, \E_{A \sim \cQ} \kr{ \frac{\int_{\cE} \int_{\cE} \cP(A,z) \, \cP(A,z') \, dz \, dz'}{ \cQ(A)^2} } = 1 + o(1) 
\end{equation*}
where the last equality holds by \eqref{eq:goal-of-conditional-second-moment} as $\cP(A,z) = \cP(A \mid z) \, \cP(z) = \cP(A \mid z)$.
It follows that
$$
\TV(\cP'_A,\cQ) \le \chi^2(\cP'_A,\cQ) = \E_{A \sim \cQ} \kr{ \pr{\frac{ \cP'}{ \cQ} (A)}^2-1 } = o(1).
$$

Combining the above bounds with the triangle inequality $\TV\pr{\cP_A,\cQ} \leq \TV\pr{\cP_A,\cP'_A} +  \TV\pr{\cP'_A,\cQ}$ then completes the proof.
\end{proof}

\subsection{Computations conditional on the underlying cycle}

To apply Lemma~\ref{lem:conditional-second-moment}, let us first compute the inner expectation in \eqref{eq:goal-of-conditional-second-moment}.

\begin{lemma}\label{lem:second moment conditioned on z}
Let $z \in [0,1]^n$ and $z' \in [0,1]^n$ be fixed. Let $X,X' \in \{0,1\}^{n \times n}$ be defined by $X_{ij} := \bbone_{\dist(z_i,z_j) \le \tau/2}$ and $X'_{ij} := \bbone_{\dist(z'_i,z'_j) \le \tau/2}$ (cf.\ Definition~\ref{def:mod-p}). 
Define $\lambda:=\frac{(p-q)^2}{r(1-r)}$.
Then we have
\begin{equation*}
\E_{A \sim \cQ} \kr{ \frac{ \cP(A \mid z) \, \cP(A \mid z') }{ \cQ(A)^2 }} \leq \exp\br{\lambda\pr{\sum_{1\leq i<j\leq n} X_{ij}'(X_{ij}-\tau) - \tau \sum_{1 \le i < j \le n} (X_{ij} - \tau)}} .
\end{equation*}
\end{lemma}

\begin{proof}
Since $X$ and $z$ are deterministic functions of each other, we have $\cP(A \mid z) = \cP(A \mid X)$.
Note that
\begin{equation*}
\frac{ \cP(A\cond X)}{\cQ(A)}            = \prod_{1\leq i<j \leq n} \pr{\frac{p X_{ij} + q (1-X_{ij})}{r}}^{A_{ij}} \pr{\frac{(1-p) X_{ij} + (1-q) (1-X_{ij})}{1-r}}^{1-A_{ij}}.
\end{equation*}
Then we take the expectation with respect to $A \sim \cQ$ (i.e., $A_{ij} \sim \Bern(r)$ i.i.d.) to obtain
\begin{equation}\label{eq:take expectation wrt A}
\begin{aligned}
& \E_{A \sim \cQ} \kr{ \frac{ \cP(A\cond X) \cP(A\cond X')}{\cQ(A)^2} } \\
=  \prod_{1\leq i< j \leq n} & \bigg( \frac{(p X_{ij} + q (1-X_{ij}))(p X_{ij}' + q (1-X_{ij}'))}{r}                                                   \\
& + \frac{((1-p) X_{ij} + (1-q) (1-X_{ij}))((1-p) X_{ij}' + (1-q) (1-X_{ij}'))}{1-r}\bigg)                               \\
=  \prod_{1\leq i<j \leq n}  & \pr{\frac{p^2}{r}+\frac{(1-p)^2}{1-r}}^{X_{ij}X_{ij}'}    
\pr{\frac{q^2}{r}+\frac{(1-q)^2}{1-r}}^{(1-X_{ij})(1-X_{ij}')}                                                           \\
& \cdot \pr{\frac{pq}{r}+\frac{(1-p)(1-q)}{1-r}}^{X_{ij}(1-X_{ij}')+X'_{ij}(1-X_{ij})} 
\end{aligned}
\end{equation}
where the last equality holds because $X_{ij}, X'_{ij} \in \{0,1\}$.
By the definition $r = \tau p + (1-\tau)q$, we have $p-r = (1-\tau)(p-q)$ and $q-r = -\tau(p-q)$. Together with the definition $\lambda = \frac{(p-q)^2}{r(1-r)}$ and the fact $1+x \le e^x$, this implies the following elementary identities
\begin{align}\label{eq:elementary identities}
\begin{aligned}
\frac{p^2}{r}+\frac{(1-p)^2}{1-r}   & = 1+\frac{(p-r)^2}{r(1-r)} \le \exp\left( \frac{(p-r)^2}{r(1-r)} \right) = \exp\left( \lambda (1-\tau)^2 \right) ,   \\
\frac{q^2}{r}+\frac{(1-q)^2}{1-r}   & = 1+\frac{(q-r)^2}{r(1-r)} \le \exp\left( \frac{(q-r)^2}{r(1-r)} \right)= \exp\left( \lambda \tau^2 \right)  , \\
\frac{pq}{r}+\frac{(1-p)(1-q)}{1-r} & = 1+\frac{(p-r)(q-r)}{r(1-r)} \le \exp\left( \frac{(p-r)(q-r)}{r(1-r)} \right) = \exp\left( - \lambda \tau (1-\tau) \right) .
\end{aligned}
\end{align}
Combining \eqref{eq:take expectation wrt A} and \eqref{eq:elementary identities}, we get
\begin{equation*}
\begin{aligned}
\E_{A \sim \cQ}\bigg[ \frac{ \cP(A\cond X) \cP(A\cond X')}{\cQ(A)^2} \bigg]
&\le \exp\bigg( \lambda \sum_{1\leq i<j \leq n} \Big( (1-\tau)^2 X_{ij}X_{ij}' + \tau^2 (1-X_{ij})(1-X_{ij}') \\
& \qquad \qquad \qquad \qquad - \tau (1-\tau) (X_{ij}(1-X_{ij}')+X'_{ij}(1-X_{ij})) \Big) \bigg) \\
&= \exp\bigg( \lambda \sum_{1\leq i<j \leq n} \Big( X_{ij}X_{ij}' - \tau (X_{ij} + X'_{ij}) + \tau^2 \Big) \bigg) ,
\end{aligned}
\end{equation*}
where the last quantity is equal to the desired bound.
\end{proof}

Note the particular form of the bound in Lemma~\ref{lem:second moment conditioned on z}: The second term in the exponent
$$
\sum_{1 \le i < j \le n} (X_{ij} - \tau)
= \sum_{1 \le i < j \le n} (\bbone_{\dist(z_i,z_j) \le \tau/2} - \tau)
$$
is a $U$-statistic with zero mean as $\p\{\dist(z_i,z_j) \le \tau/2\} = \tau$, and so is the first term
$$
\sum_{1\leq i<j\leq n} X_{ij}'(X_{ij}-\tau) = \sum_{1\leq i<j\leq n: X_{ij}' = 1} (X_{ij}-\tau)
$$
once we condition on $z'$ (i.e., on $X'$). 
Therefore, to bound the conditional second moment of the likelihood ratio, it essentially boils down to analyzing the concentration properties of these $U$-statistics, which is the most technical part of the proof.
Let us first show how Theorem~\ref{thm:detection lower bound} follows from Lemmas~\ref{lem:conditional-second-moment} and~\ref{lem:second moment conditioned on z} before analyzing the $U$-statistics.

\subsection{Proof of Theorem~\ref{thm:detection lower bound}}
\label{sec:proof-of-detect-lower}

By Lemmas~\ref{lem:conditional-second-moment} and~\ref{lem:second moment conditioned on z}, it suffices to find an event $\cE$ such that $\cP(\cE) = 1-o(1)$ and 
\begin{equation}
\E_{z,z'} \left[ \exp\br{\lambda\pr{\sum_{1\leq i<j\leq n} X_{ij}'(X_{ij}-\tau) - \tau \sum_{1 \le i < j \le n} (X_{ij} - \tau)}} \bbone_{z \in \cE, \, z' \in \cE} \right] = 1+o(1).
\label{eq:condition-mgf-small-1}
\end{equation}
To this end, we find two events, $\goodEvent_0$ defined in \eqref{eq:goodEvent_2} and $\goodEvent_1$ defined in \eqref{eq:goodEvent_1}, and let
\begin{equation}\label{eq:goodEvent}
\goodEvent := \goodEvent_0 \cap \goodEvent_1 .
\end{equation} 
By Lemmas~\ref{lem:bound_goal1} and~\ref{lem:good event}, we indeed have $\cP(\goodEvent) = 1-o(1)$. 
Furthermore, Lemmas~\ref{lem:bound_goal1} and~\ref{lem:bound_goal2} imply, respectively,
\begin{align}
& \exp\Big(-2\lambda\tau\sum_{1\leq i<j\leq n} (X_{ij} - \tau)\Big) = 1 + o(1) \quad \text{ for any } \zb \in \goodEvent_0 , \label{eq:goal1} \\
& \E_{\zb} \Big[\exp\Big(2\lambda\sum_{1\leq i<j\leq n} X_{ij}' (X_{ij}-\tau)\Big)\Big] = 1 + o(1) \quad \text{ for any } \zb'\in\goodEvent_1 . \label{eq:goal2}
\end{align}

It remains to show that \eqref{eq:condition-mgf-small-1} follows from \eqref{eq:goal1} and \eqref{eq:goal2}.
The convexity of the exponential function implies $e^{A+B} - 1 \le \frac 12 (e^{2A} + e^{2B} - 2)$ for any $A,B \in \R$.
Let $B = \lambda \sum_{1\leq i<j\leq n} X_{ij}'(X_{ij}-\tau)$ and $A = - \lambda \tau \sum_{1 \le i < j \le n} (X_{ij} - \tau)$.
Since $\cP(\cE) = 1-o(1)$, we obtain
\begin{align*}
\E_{z,z'}[e^{A+B} \bbone_{z \in \cE, \, z' \in \cE} - 1]
&= \E_{z,z'}[ (e^{A+B} - 1) \bbone_{z \in \cE, \, z' \in \cE}] + o(1) \\
&\le \frac 12 \Big( \E_{z,z'}[ (e^{2A} - 1) \bbone_{z \in \cE, \, z' \in \cE}] + \E_{z,z'}[ (e^{2B} - 1) \bbone_{z \in \cE, \, z' \in \cE}] \Big) + o(1) \\
&= \frac 12 \Big( \E_{z,z'}[ e^{2A} \bbone_{z \in \cE, \, z' \in \cE}] - 1 + \E_{z,z'}[ e^{2B} \bbone_{z \in \cE, \, z' \in \cE}] - 1 \Big) + o(1) \\
&\le \frac 12 \Big( \E_{z,z'}[ e^{2A} \bbone_{z \in \cE_0}] - 1 + \E_{z,z'}[ e^{2B} \bbone_{z' \in \cE_1}] - 1 \Big) + o(1) .
\end{align*}
The above bound is $o(1)$ by \eqref{eq:goal1} and \eqref{eq:goal2}, so \eqref{eq:condition-mgf-small-1} follows.

\subsection{Concentration of U-statistics}
\label{sec:concentration-u-statistics-three-steps}

In the rest of this section, we prove \eqref{eq:goal1} and \eqref{eq:goal2} which involve the $U$-statistics $\sum_{1\leq i<j\leq n} (X_{ij}-\tau)$ and $\sum_{1\leq i<j\leq n} X_{ij}'(X_{ij}-\tau)$.
Some useful concentration inequalities for $U$-statistics are given in Appendix~\ref{sec:u-stats}.
First, \eqref{eq:goal1} is an immediate consequence of the following lemma.

\begin{lemma}
\label{lem:bound_goal1}
Consider i.i.d.\ $z_1, \dots, z_n \sim \Unif([0,1])$ and let $X_{ij} := \bbone_{\dist(z_i,z_j) \le \tau/2}$ for $1 \le i < j \le n$. 
Suppose that the parameters satisfy $\tau \to 0$ and $n \lambda \tau^{3/2} \to 0$ as $n \to \infty$ (which is weaker than \eqref{eq:cond_detection_lower_bound}).
Then there are two sequences of positive numbers 
$t_n \to 0$ and $c_n \to 0$ as $n \to 0$ such that $\p(\cE_0) \ge 1 - c_n$ where the event $\cE_0$ is defined as
\begin{equation}\label{eq:goodEvent_2}
\goodEvent_0 := \Big\{ \zb\in [0,1]^n: \lambda \tau \Big| \sum_{1\leq i<j\leq n} (X_{ij} - \tau) \Big| \le t_n \Big\} .
\end{equation}
\end{lemma}

\begin{proof}
Let $z''$ be an independent copy of $z = (z_1, \dots, z_n)$, and define $\tilde X_{ij} := \bbone_{\dist(z_i, z''_j) \le \tau/2}$. 
By the tail bounds in Lemmas~\ref{lem:decoupling} and~\ref{lem:hti moments}, there are absolute constants $K_1, K_2 > 0$ such that for any $t>0$,
\begin{align*}
\p \Big\{\Big|\sum_{1\leq i<j\leq n} (X_{ij} - \tau) \Big| > t \Big\} 
&\le K_1 \, \p \Big\{K_1 \Big|\sum_{1\leq i<j\leq n} (\tilde X_{ij} - \tau) \Big| > t \Big\} \\
&\le K_2 \exp\br{-\frac{1}{K_2}\min\kr{\frac{t}{C},\pr{\frac{t}{B}}^{2/3},\pr{\frac{t}{A}}^{1/2}}} ,
\end{align*}
%
%
where
\begin{itemize}
\item $A = \max\{1-\tau,0-\tau\} \leq 1$;

\item $C^2 = \sum_{1\leq i<j\leq n} \E [(\tilde X_{ij}-\tau)^2] \leq n^2 \tau$;

\item $B^2 \leq n\tau$ because for any $1 \le i < j \le n$ and any $x,y \in [0,1]$, 
\begin{equation*}
\sum_{j=i+1}^n \E_{z_j''} [(\1_{\dfrak(x,z_j'')\leq \tau/2}-\tau)^2] \le n \tau, \quad \sum_{i=1}^{j-1} \E_{z_i}  [(\1_{\dfrak(z_i,y)\leq \tau/2}-\tau)^2] \leq n\tau.
\end{equation*}
\end{itemize}
Replacing $t$ by $\frac{t}{\lambda \tau}$ and plugging in the bounds on $A,B,C$, we obtain
\begin{align}
\p \Big\{\lambda \tau \Big|\sum_{1\leq i<j\leq n} (X_{ij} - \tau) \Big| > t \Big\} 
\le K_2 \exp\br{-\frac{1}{K_2}\min\kr{\frac{t}{n \lambda \tau^{3/2}}, \frac{t^{2/3}}{n^{1/3} \lambda^{2/3} \tau}, \frac{t^{1/2}}{\lambda^{1/2} \tau^{1/2}}}} . \label{eq:e-0-tail-bound}
\end{align}
By the assumptions $n \lambda \tau^{3/2} \to 0$ and $\tau \to 0$, the above bound tends to $0$ as $n \to 0$ for any fixed $t>0$ or even if $t=t_n \to 0$ sufficiently slowly. The conclusion follows.
%
\end{proof}

Next, we move on to the proof of \eqref{eq:goal2}, the most technical step of the lower bound. Let us start with some informal observations that motivate the proof.

\begin{remark}
\label{rmk:technical}
Note that for a typical $z'$, i.e., $X'$, we have $X'_{ij} = 1$ for $O(n^2 \tau)$ pairs $(i,j) \in \binom{[n]}{2}$. The variance of $X_{ij}$ is $\tau(1-\tau)$, so conditional on a typical $z'$, the variance of $\sum_{1\leq i<j\leq n} X_{ij}' (X_{ij}-\tau)$ is $O(n^2 \tau^2)$.
Consequently, the $U$-statistic in \eqref{eq:goal2} is expected to be of order
\begin{equation}
\lambda \sum_{1\leq i<j\leq n} X_{ij}' (X_{ij}-\tau) = O(n \tau \lambda) 
\label{eq:u-stats-typical-size}
\end{equation}
with high probability.
We therefore need $n \tau \lambda \to 0$ (cf.\ \eqref{eq:cond_detection_lower_bound}) for \eqref{eq:goal2} to hold.
However, there are a few obstacles in proving \eqref{eq:goal2} rigorously:
\begin{itemize}
\item
The above argument would have finished the proof if we were allowed to condition on an event $\cE$ defined in terms of the pair $(z,z')$, i.e., $(X,X')$. However, this is not allowed, because in the conditional second moment method (Lemma~\ref{lem:conditional-second-moment}), the event $\cE$ is defined in terms of one copy $z$.

\item
For a worst-case $z'$ such that $X'_{ij} = 1$ for all $1 \le i < j \le n$, we have $\lambda \sum_{1\leq i<j\leq n} X_{ij}' (X_{ij}-\tau)$ of order $n \sqrt{\tau} \lambda$, which would yield a suboptimal condition $n \sqrt{\tau} \lambda \to 0$. Therefore, we need to condition on $z' \in \cE_1$ for the event $\cE_1$ defined in \eqref{eq:goodEvent_1}.

\item
Conditional on $z' \in \cE_1$, \eqref{eq:goal2} does not follow trivially from the typical order \eqref{eq:u-stats-typical-size} because the $U$-statistic $\sum_{1\leq i<j\leq n} X_{ij}' (X_{ij}-\tau)$ can potentially be very large on a rare event. Furthermore, \eqref{eq:goal2} does not follow from standard concentration of $U$-statistics (Lemma~\ref{lem:hti moments}) either. To see this, note that the tail bound in Lemma~\ref{lem:hti moments} at best gives 
$$
\p\bigg\{ \bigg| \sum_{1\leq i<j\leq n} X_{ij}' (X_{ij}-\tau) \bigg| > t \bigg\} \le K \exp(-t^{1/2}/K)
$$
for a constant $K>0$ because $A$ is of order $1$. Taking $t$ to be of order $n^2 \tau$, we see that $\lambda \sum_{1\leq i<j\leq n} X_{ij}' (X_{ij}-\tau)$ would be of order $n^2 \tau \lambda$ with probability $K \exp(-n \sqrt{\tau} / K)$. For \eqref{eq:goal2} to hold, we would at least need $n^2 \tau \lambda \le n \sqrt{\tau}$ which gives the suboptimal condition $n \sqrt{\tau} \lambda \le 1$.
\end{itemize}
Due to the above obstacles, \eqref{eq:goal2} requires a delicate proof, where we first split the summands in $\sum_{1\leq i<j\leq n} X_{ij}' (X_{ij}-\tau)$ into a few independent groups and then make use of the independence together with moment inequalities for $U$-statistics.
\end{remark}

\paragraph{Step 1: Splitting the U-statistic into independent sums}
To study the $U$-statistic
\begin{equation}
S := \sum_{1\leq i<j\leq n} X_{ij}' (X_{ij}-\tau) ,
\label{eq:def-total-S}
\end{equation}
we will split it into three terms $S_1+S_2+S_3$, each of which is an independent sum.
Towards this end, we define
\begin{equation}
k := \max\{ i : i \text{ is an even integer} , \, i \le 1/\tau \} ,
\label{eq:def-k-1/tau}
\end{equation}
and define blocks 
\begin{equation}
B_\ell := [(\ell-1)\tau,\ell\tau) \text{ for } 1 \leq \ell < k, \quad B_{k}:= [(k-1)\tau,1) . 
\label{eq:blocks}
\end{equation}
Then $[0,1) = \bigcup_{i=1}^{k} B_i$. 
For $1 \le \ell \le k$, let
\begin{equation} \label{eq:bucket index def}
\begin{aligned} 
I_{\ell} &:= \br{i \in [n]: z_i'\in B_\ell} , \quad I_{k+1} := I_1 , \\
J_{\ell}^{1} &:= \br{(i,j) \in \binom{[n]}{2} : i, j\in I_{\ell}}, \\
J_{\ell}^{2} &:= \br{(i,j) \in \binom{[n]}{2} : i\in I_{\ell}, \, j\in I_{\ell+1} \text{ or } j\in I_{\ell}, \, i\in I_{\ell+1}} \text{ if $\ell$ is odd}, \quad J_\ell^2 := \varnothing \text{ if $\ell$ is even} , \\
J_{\ell}^{3} &:= \br{(i,j) \in \binom{[n]}{2} : i\in I_{\ell}, \, j\in I_{\ell+1} \text{ or } j\in I_{\ell}, \, i\in I_{\ell+1}} \text{ if $\ell$ is even}, \quad J_\ell^3 := \varnothing \text{ if $\ell$ is odd} .
\end{aligned}
\end{equation}
For $\ii = 1,2,3$, we define
\begin{equation}
U_\ell^{(\ii)} := \sum_{(i,j)\in J^{\ii}_{\ell}} X_{ij}' (X_{ij}-\tau) , \qquad
S_{\ii} := \sum_{\ell=1}^k U_\ell^{(\ii)} .
\label{eq:def-U-123}
\end{equation}

\begin{lemma}
\label{lem:independence of U}
In the notation above, we have $S = S_1+S_2+S_3$.
Moreover, for any fixed $z' \in [0,1]^n$ and each $1\leq \ii\leq 3$, the quantity $S_{\ii}$ is an independent sum, i.e., $U_1^{(\ii)},U_2^{(\ii)},\dots,U_k^{(\ii)}$ are independent, with respect to the randomness of $z$. 
\end{lemma}

\begin{proof}
First, by definition, $X'_{ij} = 1$ if and only if $\dist(z'_i,z'_j) \le \tau/2$. 
Therefore, if $X'_{ij}(X_{ij}-\tau)$ is a nonzero summand in $S$, then $z'_i$ and $z'_j$ are either in the same block $B_\ell$ or in adjacent blocks $B_\ell$ and $B_{\ell+1}$ for some $1\le \ell \le k$.
In view of \eqref{eq:bucket index def}, the pair $(i,j)$ must be in $J^1_\ell$, $J^1_\ell$, or $J^3_\ell$.
It then follows that $S = S_1 + S_2 + S_3$.

Next, fix $z'$ (i.e., $X'$) and consider the randomness of $z$ (i.e., $X$).
Note that for each $1 \le \ell \le k$, the random variable $U_\ell^{(1)}$ only involves $X_{ij}$ for $(i,j) \in J^1_\ell$. Hence $U_\ell^{(1)}$ is measurable with respect to $\br{z_i}_{i\in I_{\ell}}$.
Since $z_1, \dots, z_n$ are independent and $I_1, \dots, I_k$ are disjoint, we see that $U^{(1)}_1, \dots, U^{(1)}_k$ are independent.

The argument for $U^{(\ii)}_\ell$ where $i = 2,3$ is similar. Consider $U^{(2)}_\ell$.
If $\ell$ is even, then $U^{(2)}_\ell = 0$. If $\ell$ is odd, then $U_\ell^{(2)}$ is measurable with respect to $\br{z_i}_{i\in I_{\ell}\cup I_{\ell+1}}$.
Since the $k/2$ sets $I_1 \cup I_2, I_3 \cup I_4, \dots, I_{k-1} \cup I_k$ are disjoint, we see that $U^{(2)}_1, \dots, U^{(2)}_k$ are independent.
Analogously,  $U^{(3)}_1, \dots, U^{(3)}_k$ are also independent.
\end{proof}

Fix $z' \in \cE_1$ where $\cE_1$ is to be defined in \eqref{eq:goodEvent_1}, we consider the expectation with respect to the randomness of $z$. 
We use the convexity of $\exp(x)-1$ to obtain
\begin{equation}
\E_z \kr{\exp(2\lambda S) - 1} \leq \frac 13 \sum_{\ii=1}^3 \E_z \kr{\exp(6\lambda S_\ii)-1}.
\label{eq:S-123-Jensen}
\end{equation}
Hence, to prove \eqref{eq:goal2}, it suffices to bound the moment generating function of $S_\ii$ for $1\leq \ii\leq 3$.

\paragraph{Step 2: Controlling individual summands}
To study $S_\ii$ which is a sum of $k$ independent $U$-statistics, we first focus on each individual $U$-statistic $U_\ell^{(\ii)}$ for fixed $1 \le i \le 3, \, 1 \le \ell \le k$.
The number of summands in each $U_\ell^{(\ii)}$ is determined by $z'$ in view of \eqref{eq:bucket index def} and \eqref{eq:def-U-123}.
To control this, we now formally introduce the high-probability event $\cE_1$ about $z'$.

\begin{lemma}
\label{lem:good event}
For i.i.d.\ $z'_1, \dots, z'_n \sim \Unif([0,1])$, let $I_\ell$ be defined in \eqref{eq:bucket index def}.
Suppose that $\frac{n \tau}{\log n} \to \infty$ as $n \to \infty$.
Then there is a constant $n_0 > 0$ such that for all $n \ge n_0$, we have $\p(\goodEvent_1) \geq 1-n^{-10}$, where the event $\cE_1$ is defined as
\begin{equation}\label{eq:goodEvent_1}
\goodEvent_1 := \br{z' \in [0,1]^n : \abs{I_\ell} \leq 4n\tau \text{ for all } 1 \le \ell \le k}.
\end{equation}
\end{lemma}

\begin{proof}
By the definition \eqref{eq:blocks}, the largest block $B_k$ has size at most $3 \tau$ (and the other blocks have size $\tau$), the random variable $|I_\ell| = \abs{\{i \in [n] : z_i' \in B_\ell\}}$ is stochastically dominated by $\Bin(n, 3 \tau)$. 
By Bernstein's inequality, with probability at least $1-n^{-11}$, we have
\begin{equation}
\abs{I_\ell} \leq 3n\tau + C\Big( \sqrt{n\tau\log n} + \log n \Big) \leq 4 n \tau
\end{equation}
for all sufficiently large $n$ since $\frac{n\tau}{\log n} \to \infty$.
Applying a union bound finishes the proof.
\end{proof}  

In the sequel, we condition on $z' \in \cE_1$ and estimate the moments of $U^{(\ii)}_\ell$ with respect to the randomness of $z$. Let us start with the second moment.

\begin{lemma}
\label{lem:u-2-moment}
Fix $z' \in \cE_1$ and consider the randomness with respect to $z$. Then for any $1 \le \ii \le 3$ and $1 \le \ell \le k$, we have
$$
\E[(U_\ell^{(\ii)})^2] \le 16 n^2 \tau^3 .
$$
\end{lemma}

\begin{proof}
Writing $h_{ij} = X'_{ij} (X_{ij} - \tau)$ for brevity, we have $U_\ell^{(\ii)} := \sum_{(i,j)\in J^{\ii}_{\ell}} h_{ij}$ by the definition \eqref{eq:def-U-123}.
The second moment of $U_\ell^{(\ii)}$ can be computed as
\begin{equation*}
\E[(U_\ell^{(\ii)})^2]  = \E \Big[ \Big( \sum_{(i,j)\in J_\ell^{(\ii)}} h_{ij} \Big)^2 \Big] 
= \sum_{(i,j)\in J_\ell^{(\ii)}} \E \kr{h_{ij}^2} + \sum_{\substack{(i,j),(i',j')\in J_\ell^{(\ii)} : \\ i\ne i' \text{ or } j\neq j' }} \E \kr{h_{ij}h_{i'j'}} .
\end{equation*}
Note that conditioning on $z_i$, if $j \neq j'$, then $X_{ij} = \bbone_{\dist(z_i, z_j) \le \tau/2}$ and $X_{ij'} = \bbone_{\dist(z_i, z_{j'}) \le \tau/2}$ are independent. 
Therefore, $h_{ij}$ and $h_{ij'}$ are independent, and similarly for $h_{ij}$ and $h_{i'j}$ if $i \ne i'$. 
It follows that the last expectation is zero in the above equation because $\E[h_{ij}] = 0$.
We then obtain
\begin{equation}
\label{eq:g second order}
\E[(U_\ell^{(\ii)})^2] = \sum_{(i,j)\in J_\ell^{(\ii)}} \E \kr{h_{ij}^2} 
\le 16 n^2 \tau^2 \cdot \tau (1-\tau) \le 16 n^2 \tau^3 ,
\end{equation}
where we have used that $|I_\ell| \le 4 n \tau$ on the event $\cE_1$ and hence $|J^{(\ii)}_\ell| \le (4 n \tau)^2$ by the definition \eqref{eq:bucket index def}, together with the fact $\E[h_{ij}^2] \le \tau (1-\tau)$.
\end{proof}

The following lemma bounds the higher moments of $U^{(\ii)}_\ell$.

\begin{lemma}
\label{lem:u-higher-moment}
Fix $z' \in \cE_1$ and consider the randomness with respect to $z$. Then there is an absolute constant $K>0$ such that for any $1 \le \ii \le 3$, $1 \le \ell \le k$, and $d \ge 3$, we have
$$
\E[|U_\ell^{(\ii)}|^d] \le K^d \Big[ (n^2 \tau^2)^d \land \Big( d^d (n^2 \tau^3)^{d/2} + d^{3d/2} (n \tau^2)^{d/2} + d^{2d}(n^2 \tau^3 \land 1) \Big) \Big] .
$$
\end{lemma}

\begin{proof}
First of all, we know that $|J^{(\ii)}_\ell| \le 16 n^2 \tau^2$ as $|I_\ell| \le 4 n \tau$ on the event $\cE_1$. By the definition of $U_\ell^{(\ii)}$ in \eqref{eq:def-U-123}, it is trivial that $|U_\ell^{(\ii)}|^d \le (16 n^2 \tau^2)^d$.

Let $z''$ be an independent copy of $z$. Define a decoupled version of $X'_{ij}(X_{ij} - \tau)$ by $$
\tilde h_{ij} := X'_{ij} (\bbone_{\dist(z_i, z''_j)} - \tau) .
$$
For $d\geq 3$, we apply Lemma~\ref{lem:decoupling} with $f_{ij} = \hti_{ij} \cdot \bbone_{(i,j) \in J^{(\ii)}_\ell}$ and $\Psi(x) = x^d$ to obtain
\begin{equation*}
\E \kr{\abs{U^{(\ii)}_\ell}^d} \leq 12^d \, \E \bigg[ \Big|\sum_{(i,j)\in J_{\ell}^{(\ii)}} \hti_{ij}\Big|^d \bigg] .
\end{equation*}

Furthermore, the moments of the decoupled $U$-statistic can be controlled using Lemma~\ref{lem:hti moments}. 
Let us compute all the quantities involved:
\begin{itemize}
\item $ d^d\pr{\sum_{(i,j)\in J_{\ell}^\ii} \E \hti_{ij}^2 }^{d/2} 
\le d^d \pr{16 n^2 \tau^3}^{d/2}$ by the same argument as that leading to \eqref{eq:g second order}.

\item $ d^{3d/2}\E_{\zb} \max_{i\in I_{\ell}}\pr{\sum_{j\in I_{\ell}} \E_{z''} \hti_{ij}^2}^{d/2} \leq d^{3d/2} \pr{|I_\ell| \tau}^{d/2} \le d^{3d/2} \pr{4 n\tau^2}^{d/2}$ again by the bound $|I_\ell| \le 4 n \tau$ on the event $\cE_1$ and that $\E[\tilde h_{ij}^2] = \tau (1-\tau)$. The same bound obviously holds if $i$ or $j$ belongs to $I_{\ell+1}$ instead of $I_\ell$, or if $i$ and $j$ are swapped.

\item 
$d^{2d} \E \max_{(i,j)\in J_{\ell}^\ii}\abs{\hti_{ij}}^d \le d^{2d} \pr{16 n^2\tau^3 \land 1}$ for the following reason. 
First, $\max_{(i,j)\in J_{\ell}^\ii}\abs{\hti_{ij}}^d = (1-\tau)^d$ if there exists $(i,j) \in J^\ii_\ell$ for which $\dist(z_i,z''_j) \le \tau/2$; this happens with probability at most $|J^\ii_\ell| \tau \le 16 n^2 \tau^3$.
Otherwise, $\max_{(i,j)\in J_{\ell}^\ii}\abs{\hti_{ij}}^d = \tau^d$.
Therefore, $\E \max_{(i,j)\in J_{\ell}^\ii} \abs{\hti_{ij}}^d \leq 16 n^2 \tau^3 (1-\tau)^d + \tau^d \leq 17 n^2 \tau^3$ for $d \ge 3$.
Second, it is trivial that $\E \max_{(i,j)\in J_{\ell}^\ii} \abs{\hti_{ij}}^d \le 1$.

\end{itemize}
Applying Lemma~\ref{lem:hti moments} then yields the desired bound.
\end{proof}

\paragraph{Step 3: Bounding the moment generating function}

We are ready to prove \eqref{eq:goal2}.

\begin{lemma}
\label{lem:bound_goal2}
Let $S$ be defined by \eqref{eq:def-total-S} and $\goodEvent_1$ be the event defined in \eqref{eq:goodEvent_1}.
Assume condition \eqref{eq:cond_detection_lower_bound} and that $n \tau \ge (\log n)^3$. Then for any fixed $z' \in \cE_1$, we have $\E_z \kr{\exp(2\lambda S)} = 1 + o(1)$. 
\end{lemma}

\begin{proof}
%
By Lemma~\ref{lem:independence of U} and \eqref{eq:S-123-Jensen}, it suffices to prove that 
$$
\E_z \kr{\exp(6\lambda S_\ii)} = 1+o(1)
$$
for $1 \le \ii \le 3$, where $S_\ii$ is an independent sum of $U$-statistics as defined in \eqref{eq:def-U-123}. 
Note that $\E_z[ U_\ell^{(\ii)} ] = 0$. Since $U^{(\ii)}_1, \dots U^{(\ii)}_k$ are independent and finite, we use the dominated convergence theorem to obtain
\begin{equation*}
\E_z [\exp(6\lambda S_\ii)]
= \prod_{\ell=1}^k \E_z [\exp(6\lambda U_\ell^{(\ii)})]
= \prod_{\ell=1}^k \Bigg( 1 + \sum_{d=2}^\infty \frac{{6^d\lambda^d \E_z [(U_\ell^{(\ii)})^d]}}{d! } \Bigg) .
\end{equation*}
Since $k \le 1/\tau$ by the definition \eqref{eq:def-k-1/tau} and $1+x \le e^x$ for $x \in \R$, it follows that
\begin{align*}
\E_z [\exp(6\lambda S_\ii)] 
\leq \pr{1 + \max_{1\leq \ell \leq k}  \sum_{d=2}^\infty \frac{{6^d\lambda^d \E_z [|U_\ell^{(\ii)}|^d] }}{d! } }^{1/\tau} 
\leq \exp\br{\max_{1\leq \ell \leq k}  \sum_{d=2}^\infty \frac{{6^d\lambda^d \E_z [|U_\ell^{(\ii)}|^d]  }}{\tau d! }  } .
\end{align*}
It remains to show that the exponent is $o(1)$,
which is true if 
\begin{equation*}
\max_{\substack{1\leq \ell \leq k,\\ d\geq 2}} \pr{\frac{{\lambda^d \E_z [|U_\ell^{(\ii)}|^d]}}{\tau d! }}^{1/d} = o(1) .
\end{equation*}

Fix any $\ell \in [k]$. 
For $d = 2$, by Lemma~\ref{lem:u-2-moment}, we indeed have $\frac{\lambda^2}{\tau} \E[(U_\ell^{(\ii)})^2] \le 16 \lambda^2 n^2 \tau^2 = o(1)$ under condition \eqref{eq:cond_detection_lower_bound}.
For $d \ge 3$, by Lemma~\ref{lem:u-higher-moment} and Stirling's approximation, we have
\begin{align*}
\pr{\frac{{\lambda^d \E_z [|U_\ell^{(\ii)}|^d]}}{\tau d! }}^{1/d}
\le \frac{e \lambda K}{\tau^{1/d} d} \Big[ (n^2 \tau^2) \land \Big( d (n^2 \tau^3)^{1/2} + d^{3/2} (n \tau^2)^{1/2} + d^{2}(n^2 \tau^3 \land 1)^{1/d} \Big) \Big] .
\end{align*}
By condition \eqref{eq:cond_detection_lower_bound}, we have $\lambda = o(\frac{1}{n\tau})$. Hence, it suffices to verify
\begin{align*}
&\max_{d \ge 2} \frac{1}{d n \tau^{1+1/d}} \Big[ (n^2 \tau^2) \land \Big( d (n^2 \tau^3)^{1/2} + d^{3/2} (n \tau^2)^{1/2} + d^{2}(n^2 \tau^3 \land 1)^{1/d} \Big) \Big] \\
&= \max_{d \ge 2} \Big[ (n \tau^{1-1/d} d^{-1}) \land \Big( \tau^{1/2-1/d} + d^{1/2} n^{-1/2} \tau^{-1/d} + (d n^{2/d-1} \tau^{2/d-1}) \land (d n^{-1} \tau^{-1-1/d}) \Big) \Big] 
= O(1) .
\end{align*}
Thanks to the first term, we can assume $n \tau^{1-1/d} d^{-1} \ge 1$, i.e., $d \le n \tau^{1-1/d}$. 
Moreover, we have:
\begin{itemize}
\item $\tau^{1/2-1/d} \leq 1$.

\item $d^{1/2}n^{-1/2}\tau^{-1/d} \le \tau^{1/2 - 3/(2d)} \leq 1$ since $3 \le d \le n \tau^{1-1/d}$.

\item
$(d n^{2/d-1} \tau^{2/d-1}) \land (d n^{-1} \tau^{-1-1/d}) \le (d (n\tau)^{-1/3}) \land \tau^{-2/d} \le \frac{d}{\log n} \land n^{2/d}$ since $3 \le d \le n \tau^{1-1/d}$ and $n \tau \ge (\log n)^3$. 
We have either $\frac{d}{\log n} \le 1$, or $d \ge \log n$ in which case $n^{2/d} \le e^2$.
\end{itemize}
The conclusion follows.
\end{proof}

\section{Lower bound for recovery}
\label{sec:lower-bound-recovery}

We now consider the recovery problem from Definition~\ref{def:prob-recover}.
Define the minimum mean squared error (MMSE) as 
\begin{equation}
\mmse := \min_{\hat X = \hat X(A)} \E_{(A,X) \sim \cP} [R(\hat X, X)] , \quad \text{ where } R(\hat X, X) := \sum_{1 \le i < j \le n} (\hat X_{ij} - X_{ij})^2 ,
\label{def:bayes estimator, mmse}
\end{equation}
where the minimum is taken over all estimators $\hat X$ measurable with respect to the observation $A$.
To establish an impossibility result for recovery, it suffices to prove a lower bound on the MMSE.

\begin{theorem}
\label{thm:recovery lower bound}
Consider the model $\cP$ from Definition~\ref{def:mod-p} with parameters $n, \tau, p, q$.
Define $\lambda := \frac{(p-q)^2}{r(1-r)}$.
Suppose that $\frac{\log n}{n} \le r \le \frac 12$, $(\log n)^3 \le n \tau \le o(n)$, $n \tau \lambda \to 0$, and $n \lambda \tau^{3/2} \log n \to 0$.
The MMSE defined in \eqref{def:bayes estimator, mmse} satisfies 
\begin{equation*}
\mmse \geq \binom{n}{2}\tau(1-\tau) - o(n^2\tau) .
\end{equation*}
\end{theorem}

To prove the negative result for recovery, we relate the MMSE to a few information measures between $A$ and $X$, including the mutual information, the conditional entropy, and the Kullback--Leibler (KL) divergence. 
This approach has been successful at establishing negative results for recovery problems on random graphs \cite{deshpandeAsymptoticMutualInformation2016,wuSettlingSharpReconstruction2022}. Our proof follows the same line.

\subsection{Information measures}
Let 
$\KL(\cP\|\cQ)=\int\cP\log\frac{\cP}{\cQ}$
denote the KL divergence between distributions $\cP$ and $\cQ$.
For random variables $A$ and $X$, denoting the joint distribution by $\cP_{(A,X)}$ and the product distribution by $\cP_A \otimes \cP_X$, the mutual information between $A$ and $X$ is defined to be
$$
I(A,X) := \KL(\cP_{(A,X)} \| \cP_A \otimes \cP_X).
$$
Moreover, the conditional entropy of $X$ given $A$ is defined by 
$$
H(X \cond A) := - \int \cP(A,X) \log \cP(X \cond A).
$$
These quantities measure the amount of information about $X$ contained in the observation $A$.


\begin{definition}[Model $\cP_\theta$]\label{def:interpolation model}
Fix parameters $0 < r \le \theta \le 1$ and $0 < \tau < 1/2$. 
Suppose that $r > \tau \theta$, and let $q := \frac{r-\tau \theta}{1-\tau}$. 
Then we define $\cP_\theta$ in the same way as the model $\cP$ in Definition~\ref{def:mod-p} but with $p$ replaced by $\theta$.
Let $X$ and $A$ be also from Definition~\ref{def:mod-p}. We write $(A,X) \sim \cP_\theta$ and $A \sim (\cP_\theta)_A$.
\end{definition}

\noindent
With the above definition, note that $\cP_p = \cP$ and $(\cP_r)_A = \cQ$ for $\cP$ and $\cQ$ given in Definitions~\ref{def:mod-p} and~\ref{def:mod-q} respectively, so $\cP_\theta$ interpolates between $\cP$ and $\cQ$.
Moreover, for any $\theta \in [r,p]$, we have marginally $A_e \sim \Bern(r)$ for every $e \in \binom{[n]}{2}$ because of the relation $q = \frac{r-\tau \theta}{1-\tau}$, i.e., $r = \theta \tau + q (1-\tau)$.
%
%
%

We use $H_\theta(X\cond A)$ to denote the conditional entropy of $X$ given $A$ the under model $\cP_\theta$, and use $I(A,X)$ without a subscript to denote the mutual information under model $\cP = \cP_p$.
Note that for $\theta = r$, $A$ and $X$ are independent under $\cP_r = \cQ$, so $H_r(X \cond A)$ is the entropy of $X$.
It is straightforward to obtain the following standard facts:
\begin{equation}\label{eq:mutal info and conditional entropy}
-I(A,X) = H_p(X\cond A) - H_r (X\cond A) = \int_r^p \frac{\rmd }{\rmd \theta} H_{\theta} (X\cond A)\rmd \theta ,
\end{equation}
%
\begin{equation} \label{eq:mutual info lower bound}
I(A,X) = \E_X [\KL(\cP_{A\cond X} \| \cQ)] - \KL(\cP_A \| \cQ).
\end{equation}


For a pair of indices $e = (i,j) \in \binom{[n]}{2}$, let $A_{-e}$ denote $A$ without its entry $A_e$.
The following result is shown by \cite{measson2009generalized,deshpandeAsymptoticMutualInformation2016} (cf.\ Lemma~7.1 of \cite{deshpandeAsymptoticMutualInformation2016} and its proof).

\begin{lemma}[Lemma~7.1 of \cite{deshpandeAsymptoticMutualInformation2016}] \label{lem:immse}
For $e \in \binom{[n]}{2}$ and $x_e, y_e \in \{0,1\}$, let 
$$
p_\theta(y_e \cond x_e) : = \cP_\theta \br{A_e = y_e \cond X_e = x_e}.
$$
We have
\begin{equation*}
\frac{\rmd}{\rmd \theta} H_\theta(X\cond A) 
= \sum_{e\in \binom{[n]}{2}} \pr{ \frac{\rmd }{\rmd \theta} H_{\theta} (A_e \cond X_e) + \rone_e } ,
\end{equation*}
where 
\begin{align*}
\rone_e := \sum_{x_e, y_e \in \{0,1\}} \frac{\rmd p_\theta(y_e \cond x_e)}{\rmd \theta} \E \Big[ \cP_\theta\br{X_e=x_e \cond A_{-e}} \log \Big( \sum_{x_e' \in \{0,1\}} p_\theta(y_e \cond x_e') \, \cP_\theta\{X_e=x_e'\cond A_{-e}\} \Big) \Big].
\end{align*}
\end{lemma}

Applying the above results, we now prove the following.

\begin{lemma} \label{lem:kl-integral}
Let $\rone_e$ be defined in Lemma~\ref{lem:immse}.
We have 
$$
\KL(\cP_A \| \cQ) = \sum_{e\in \binom{[n]}{2}} \int_r^p \rone_e \, \rmd \theta .
$$
\end{lemma}

\begin{proof}
By \eqref{eq:mutal info and conditional entropy} and Lemma~\ref{lem:immse}, we obtain
\begin{align}
-I(A,X) 
&= \int_r^p \sum_{e\in \binom{[n]}{2}} \pr{ \frac{\rmd }{\rmd \theta} H_{\theta} (A_e \cond X_e) + \rone_e } \, \rmd \theta \notag \\
&= \sum_{e\in \binom{[n]}{2}} \pr{ \pr{H_{p} (A_e \cond X_e)- H_{r} (A_e \cond X_e)} + \int_r^p \rone_e \, \rmd \theta } . \label{eq:negative-mutual-info}
\end{align}
It remains to study the first term. Under model $\cP_r = \cQ$, we have that $A_e$ and $X_e$ are independent, so $H_r(A_e \cond X_e)$ is the entropy of $A_e$.
It follows that
\begin{align*}
&-\pr{H_{p} (A_e \cond X_e)- H_{r} (A_e \cond X_e)}\\ 
&= \sum_{x_e, y_e \in \{0,1\}} \cP\pr{A_e = y_e, \, X_e = x_e} \log \cP\pr{A_e = y_e\cond X_e} - \sum_{y_e \in \{0,1\}} \cQ\pr{A_e = y_e} \log \cQ(A_e = y_e)\\
&= 
\sum_{x_e, y_e \in \{0,1\}} \cP\pr{A_e = y_e, \, X_e = x_e} 
\log \frac{\cP\pr{A_e = y_e\cond X_e}}{\cQ(A_e = y_e)} ,
\end{align*}
where the last step holds because $\sum_{x_e \in \{0,1\}} \cP\pr{A_e = y_e, \, X_e = x_e} = \cP\pr{A_e = y_e} = r = \cQ\pr{A_e = y_e}$.
As a result,
\begin{align*}
- \sum_{e\in \binom{[n]}{2}} \pr{H_{p} (A_e \cond X_e)- H_{r} (A_e \cond X_e)} 
&= \sum_{e\in \binom{[n]}{2}} \E_{X_e} \sum_{y_e \in \{0,1\}} \cP\pr{A_e = y_e \cond X_e} 
\log \frac{\cP\pr{A_e = y_e\cond X_e}}{\cQ(A_e = y_e)} \\
&= \sum_{e\in \binom{[n]}{2}} \E_{X_e}[ \KL( \cP_{A_e|X_e} \| \cQ_{A_e} ) ] .
\end{align*}
Note that $A_e \cond X_e = A_e \cond X$ and the entries $(A_e)_{e \in \binom{[n]}{2}}$ are independent conditional on $X$ under $\cP$.  Hence, the above sum is equal to $\E_X[\KL(\cP_{A\cond X} \| \cQ)]$. We then obtain from \eqref{eq:negative-mutual-info} that
$$
I(A,X) = \E_X[\KL(\cP_{A\cond X} \| \cQ)] - \sum_{e\in \binom{[n]}{2}} \int_r^p \rone_e \, \rmd \theta .
$$
This combined with \eqref{eq:mutual info lower bound} completes the proof.
\end{proof}

\subsection{Relation to the MMSE}
Fix $\theta \in [r,p]$ and consider the model $\cP_\theta$ in Definition~\ref{def:interpolation model}.
We relate the quantity $\sum_{e\in \binom{[n]}{2}} \rone_e$ that appears in Lemma~\ref{lem:kl-integral} to the MMSE defined as (cf.\ \eqref{def:bayes estimator, mmse} which corresponds to $\theta = p$)
\begin{equation}
\mmse(\theta) := \min_{\hat X = \hat X(A)} \E_\theta[R(\hat X, X)] ,
\label{eq:mmse-theta}
\end{equation}
where $(A,X) \sim \cP_\theta$ and $\hat X$ is an estimator of $X$ that is measurable with respect to $A$.

\begin{lemma} \label{lem:mutual info upper bound 1}
For $e \in \binom{[n]}{2}$, define $\rone_e$ as in Lemma~\ref{lem:immse} and define
$$
\Xhat_e := \E_\theta[X_e \cond A_{-e}].
$$
Let $\Var_\theta(\Xhat_e)$ denote the variance of $\Xhat_e$ under $\cP_\theta$. 
We have
\begin{equation*}
\Var_\theta(\Xhat_e) \le \frac{(1-\tau)^2 \theta}{\theta-r} \, \rone_e .
\end{equation*}
\end{lemma}

\begin{proof}
Recall the model $\cP_\theta$ in Definition~\ref{def:interpolation model} and the notation $p_\theta(y_e \cond x_e)$ in Lemma~\ref{lem:immse}. We have
$$
p_\theta(y_e \cond x_e) =
\begin{cases}
\theta & \text{if } x_e = y_e = 1, \\
1-\theta, & \text{if } x_e = 1, \, y_e = 0, \\
q = \frac{r-\tau\theta}{1-\tau}, & \text{if } x_e = 0, \, y_e = 1, \\
1-q  & \text{if } x_e = y_e = 0,
\end{cases}
$$
so
\begin{equation*}
\frac{\rmd p_\theta(y_e\cond x_e)}{\rmd \theta} = 
\begin{cases}
1, & \text{if } x_e = y_e = 1, \\
-1, & \text{if } x_e = 1, \, y_e = 0, \\
-\frac{\tau}{1-\tau} & \text{if } x_e = 0, \, y_e = 1, \\
\frac{\tau}{1-\tau} & \text{if } x_e = y_e = 0. 
\end{cases}
\end{equation*}
Let $h(x) := -x\log x-(1-x)\log(1-x)$ and then $h'(x) = \log\frac{1-x}{x}$.
Moreover, define
\begin{equation*}
\dot X_e := \sum_{x_e' \in \{0,1\}} p_\theta(1 \cond x_e') \, \cP_\theta\{X_e=x_e'\cond A_{-e}\} = \theta \Xhat_e + \frac{r-\tau\theta}{1-\tau}(1-\Xhat_e)
\end{equation*}
and then
$$
\sum_{x_e' \in \{0,1\}} p_\theta(0 \cond x_e') \, \cP_\theta\{X_e=x_e'\cond A_{-e}\} 
= 1 - \dot X_e .
$$

Therefore, for each $x_e \in \{0,1\}$, we have 
\begin{align*}
& \sum_{y_e \in \{0,1\}} \frac{\rmd p_\theta(y_e \cond x_e)}{\rmd \theta} \E \Big[ \cP_\theta\br{X_e=x_e \cond A_{-e}} \log \Big( \sum_{x_e' \in \{0,1\}} p_\theta(y_e \cond x_e') \, \cP_\theta\{X_e=x_e'\cond A_{-e}\} \Big) \Big] \\
&= \frac{\rmd p_\theta(1 \cond x_e)}{\rmd \theta} \E \Big[ \cP_\theta\br{X_e=x_e \cond A_{-e}} \log \dot X_e \Big] + \frac{\rmd p_\theta(0 \cond x_e)}{\rmd \theta} \E \Big[ \cP_\theta\br{X_e=x_e \cond A_{-e}} \log(1 - \dot X_e) \Big] \\
&= -\frac{\rmd p_\theta(1\cond x_e)}{\rmd \theta}\E \kr{\cP_\theta\br{X_e=x_e \cond A_{-e}} h'(\dot X_e)} ,
\end{align*}
where the last step follows from the facts $\frac{\rmd p_\theta(1\cond x_e)}{\rmd \theta} = - \frac{\rmd p_\theta(0\cond x_e)}{\rmd \theta}$ and $h'(x) = \log \frac{1-x}{x}$.
Summing the above equation over $x_e \in \{0,1\}$ and recalling the definition of $\rone_e$ in Lemma~\ref{lem:immse}, we obtain
\begin{equation*}
\rone_e = - \E \kr{ \hat X_e \, h'(\dot X_e)} + \frac{\tau}{1-\tau} \E \kr{ (1-\hat X_e) \, h'(\dot X_e)}
= \frac{-1}{1-\tau} \E \kr{ \pr{\Xhat_e - \tau} \, h'(\dot X_e) } .
\end{equation*}

Furthermore, let
\begin{equation*}
\Delta_e := \dot X_e - r = \frac{\theta-r}{1-\tau} \pr{\Xhat_e - \tau} .
\end{equation*}
Then we have $\E[\Delta_e]=0$ and $\E \kr{\Delta_e^2} = \pr{\frac{\theta-r}{1-\tau}}^2 \Var\pr{\Xhat_e}$. 
Since $h''(x)= - \frac{1}{x(1-x)}$, Taylor's theorem gives $h'(\dot X_e) = h'(r) - \frac{1}{\xi(1-\xi)}\Delta_e$ for some $\xi$ between $r$ and $\dot X_e$. 
Hence,
\begin{align*}
\rone_e
= \frac{-1}{1-\tau} \E \kr{ \pr{\Xhat_e - \tau} \, h'(\dot X_e) }
&= \frac{1}{(1-\tau)\xi(1-\xi)} \E \kr{ \Delta_e\pr{\Xhat_e - \tau}} \\
&= \frac{\theta-r}{(1-\tau)^2\xi(1-\xi)}\Var(\Xhat_e)
\geq \frac{\theta-r}{(1-\tau)^2 \theta}\Var(\Xhat_e),
\end{align*}
where the last bound holds because $\dot X_e$ is between $\theta$ and $q = \frac{r - \tau \theta}{1 - \tau}$ by definition and then $\xi \le \theta$.
%
\end{proof}

\begin{lemma} \label{lem:mutual info upper bound 2}
Consider the posterior mean under $\cP_\theta$ denoted by
$$
\Xtd := \E_\theta[X \cond A].
$$
Let $\Xhat$ be defined as in Lemma~\ref{lem:mutual info upper bound 1}.
Recall that $q := \frac{r-\tau \theta}{1-\tau}$.
For every $e \in \binom{[n]}{2}$, we have
\begin{align*}
\Var_\theta (\Xtd_e) &\le \Var_\theta (\Xhat_e) + 2 \theta \tau , \\
\Var_\theta (\Xtd_e) &\le \frac{\theta^2}{q^2} \Var_\theta(\Xhat_e) + \tau^2 \Big( \frac{\theta^2}{q^2} - 1 \Big) .
\end{align*}
\end{lemma}

\begin{proof}
For $e \in \binom{[n]}{2}$, we have
\begin{align*}
p_\theta\pr{X_e \cond A} 
&= \frac{ p_\theta\pr{X_e , A_e,A_{-e}}}{ p_\theta\pr{A_{-e},A_e}} \\
&= \frac{ p_\theta\pr{A_{-e}} p_\theta\pr{X_e  \cond A_{-e}} p_\theta\pr{A_{e}\cond X_e ,A_{-e}}}{\sum_{x_e\in\{0,1\}}  p_\theta\pr{A_{-e}} p_\theta\pr{x_e \cond A_{-e}} p_\theta\pr{A_{e}\cond X_e = x_e,A_{-e}}} \\
&= \frac{ p_\theta\pr{X_e  \cond A_{-e}} p_\theta\pr{A_{e}\cond X_e}}{\sum_{x_e\in\{0,1\}}  p_\theta\pr{x_e \cond A_{-e}} p_\theta\pr{A_{e}\cond X_e = x_e}} ,
\end{align*}
where the last equality holds because $p_\theta(A_e \cond X_e = x_e, A_{-e}) = p_\theta(A_e \cond X_e = x_e)$.
Recall that $\Xtd_e = p_\theta(X_e = 1 \cond A)$ and $\Xhat_e = p_\theta(X_e = 1 \cond A_{-e})$ by definition, so
\begin{align*}
\Xtd_e 
& = \frac{\Xhat_e p_\theta\pr{A_e\cond X_e = 1}}{\Xhat_e p_\theta\pr{A_e\cond X_e = 1}+(1-\Xhat_e) p_\theta\pr{A_e\cond X_e = 0}} \\
& = \frac{\Xhat_e \pr{A_e \theta + (1-A_e)(1-\theta)}}{\Xhat_e \pr{A_e \theta + (1-A_e)(1-\theta)} + (1-\Xhat_e) \pr{A_e \frac{r-\tau \theta}{1-\tau} + (1-A_e)\pr{1-\frac{r-\tau \theta}{1-\tau}}}} .
\end{align*}
If $A_e = 1$, then
$$
\Xtd_e = \frac{\Xhat_e \theta(1-\tau)}{\Xhat_e(\theta-r)+r-\tau\theta }
\le 1 \land \frac{\Xhat_e \theta(1-\tau)}{r-\tau\theta} 
= 1 \land \frac{\Xhat_e \theta}{q} 
$$
since $0 \le \hat X_e \le 1$, $r \le \theta$, and $q = \frac{r-\tau \theta}{1-\tau}$;
on the other hand, if $A_e = 0$, then 
$$
\Xtd_e =  \frac{\Xhat_e (1-\theta - \tau + \tau\theta)}{(r-\theta)\Xhat_e + (1-\tau - r + \tau \theta)}
\le \frac{\Xhat_e (1-\theta - \tau + \tau\theta)}{(r-\theta) + (1-\tau - r + \tau \theta)} 
= \Xhat_e .
$$
Combining the two cases yields two bounds:
\begin{align}
\Xtd_e^2 &\le A_e \Big( 1 \land \frac{\Xhat_e \theta}{q} \Big)^2 + (1-A_e) \Xhat_e^2 
\le A_e \Big( 1 \land \frac{\Xhat_e \theta}{q} \Big) + \Xhat_e^2 , \label{eq:x-tilde-x-hat-bd-1} \\
\Xtd_e^2 &\le \frac{\theta^2}{q^2} \hat X_e^2 . \label{eq:x-tilde-x-hat-bd-2}
\end{align}

Taking the expectation of \eqref{eq:x-tilde-x-hat-bd-1}, we obtain
$$
\E_\theta[\Xtd_e^2] 
\le \E_\theta\Big[ A_e \Big( 1 \land \frac{\Xhat_e \theta}{q} \Big) \Big] + \E_\theta[\Xhat_e^2] .
$$
Conditional on $X_e$, we have that $A_e$ is independent of $A_{-e}$ and thus of $\Xhat_e$. 
Moreover, $\E_\theta[A_e \cond X_e] = \theta X_e + q (1-X_e)$ under model $\cP_\theta$ from Definition~\ref{def:interpolation model}.  
Therefore,
\begin{align*}
\E_\theta\Big[ A_e \Big( 1 \land \frac{\Xhat_e \theta}{q} \Big) \Big] 
= \E_\theta\Big[ \big( \theta X_e + q (1-X_e) \big) \Big( 1 \land \frac{\Xhat_e \theta}{q} \Big) \Big] 
\le \E_\theta[\theta X_e] + \E_\theta[\theta \Xhat_e] = 2 \theta \tau ,
\end{align*}
where the inequality holds because $\theta X_e (1 \land \frac{\Xhat_e \theta}{q}) \le \theta X_e$ and $q(1-X_e) (1 \land \frac{\Xhat_e \theta}{q}) \le \theta \hat X_e$.
%
Finally, note that $\E_\theta [\Xtd_e] = \E_\theta [\hat X_e] = \tau$.
Putting it all together, we conclude that
\begin{align*}
\Var_\theta (\Xtd_e)
= \E_\theta[\Xtd_e^2] - \tau^2
\le 2 \theta \tau + \E_\theta[\Xhat_e^2] - \tau^2
= 2 \theta \tau + \Var_\theta(\Xhat_e) .
\end{align*}
For the other bound, we take the expectation of \eqref{eq:x-tilde-x-hat-bd-2} to obtain
$$
\Var_\theta (\Xtd_e)
= \E_\theta[\Xtd_e^2] - \tau^2
\le \frac{\theta^2}{q^2} \, \E_\theta[\Xhat_e^2] - \tau^2
= \frac{\theta^2}{q^2} \Var_\theta(\Xhat_e) + \tau^2 \Big( \frac{\theta^2}{q^2} - 1 \Big) .
$$
\end{proof}

\begin{lemma}
\label{lem:mmse-monotone}
The MMSE $\mmse(\theta)$ defined in \eqref{eq:mmse-theta} is non-increasing in $\theta \in [r,1]$. 
\end{lemma}

\begin{proof}
Fix $0 < r \le \theta < \theta' \le 1$.
To prove that $\mmse(\theta') \le \mmse(\theta)$, it suffices to show that given $A' \sim \cP_{\theta'}$, we can obtain $A \sim \cP_{\theta}$. To this end, conditional on $A'$, we sample $A_{ij} \sim \Bern(x)$ if $A'_{ij} = 1$ and $A_{ij} \sim \Bern(y)$ if $A'_{ij} = 0$, independently for $(i,j) \in \binom{[n]}{2}$. If $A'_{ij} \sim \Bern(\theta')$, then $A_{ij} \sim \Bern(x \theta' + y (1-\theta'))$ marginally; if $A'_{ij} \sim \Bern(r)$, then $A_{ij} \sim \Bern(x r + y (1-r))$. It remains to choose $x = \frac{\theta - r + r (\theta' - \theta)}{\theta' - r}$ and $y = \frac{r (\theta' - \theta)}{\theta' - r}$ so that $x \theta' + y (1-\theta') = \theta$ and $x r + y (1-r) = r$.
\end{proof}

\begin{lemma} \label{lem:mmse-trivial-interpolation}
For any $r < p' < p \le 1$, let $\mmse(p')$ be defined by \eqref{eq:mmse-theta}. Consider the models $\cP$ and $\cQ$ from Definitions~\ref{def:mod-p} and~\ref{def:mod-q} respectively. 
Then we have
\begin{align*}
\binom{n}{2} \tau(1-\tau) - \mmse(p') &\le \Big( \KL(\cP_A \| \cQ) + \frac{n^2 \tau (p-r)^2}{2 (1-\tau)^2} \Big) \frac{(1-\tau)^2 p}{(p-p')(p'-r)} , \\
\binom{n}{2} \tau(1-\tau) - \mmse(p') &\le \Big( \KL(\cP_A \| \cQ) + \frac{n^2 \tau^2 (p^2 - q^2) (p-r) (p-p')}{(1-\tau)^2 (p')^3} \Big) \frac{(1-\tau)^2 p^3}{q^2 (p-p') (p'-r)} .
\end{align*}
\end{lemma}

\begin{proof}
It is standard that the posterior mean $\tilde X$ achieves the MMSE, and, by the law of total expectation, $\E[\langle X_e, \Xtd_e \rangle] = \E[\Xtd_e^2]$ and $\E_\theta [X_e] = \E_\theta [\Xtd_e] = \tau$. We then derive
\begin{align*}
\mmse(\theta) &= \sum_{e\in\binom{[n]}{2}} \E_\theta \kr{\pr{X_e-\Xtd_e}^2} 
= \sum_{e\in\binom{[n]}{2}} \pr{\E_\theta\kr{X_e^2}-\E_\theta\kr{\Xtd_e^2}} 
= \sum_{e\in\binom{[n]}{2}} \pr{\Var_\theta\pr{X_e}-\Var_\theta\pr{\Xtd_e}} .
\end{align*}
We have $\Var_\theta (X_e) = \tau(1-\tau)$. Moreover, Lemma~\ref{lem:mutual info upper bound 1} and the first bound in Lemma~\ref{lem:mutual info upper bound 2} together imply that $\Var_\theta(\Xtd_e) \le \Var_\theta (\Xhat_e) + 2 \theta \tau \le \frac{(1-\tau)^2 \theta}{\theta-r} \, \rone_e + 2 \theta \tau$.
Therefore, 
\begin{align*}
\mmse(\theta) 
&\geq \binom{n}{2} \tau(1-\tau) - n^2 \theta \tau - \frac{(1-\tau)^2 \theta}{\theta-r} \sum_{e\in\binom{[n]}{2}} \rone_e . 
\end{align*}
%
%
Combining this with Lemma~\ref{lem:kl-integral} then gives
\begin{align*}
\KL(\cP_A \| \cQ) = \sum_{e\in \binom{[n]}{2}} \int_r^p \rone_e \, \rmd \theta 
& \ge \int_r^p \frac{\theta-r}{(1-\tau)^2 \theta} \bigg[ \binom{n}{2} \tau(1-\tau) - \mmse(\theta) - n^2 \tau \theta \bigg] \, \rmd \theta \\
& = \int_{r}^p \frac{\theta - r}{(1-\tau)^2 \theta} \bigg[ \binom{n}{2} \tau(1-\tau) - \mmse(\theta) \bigg] \, \rmd \theta - \frac{n^2 \tau (p-r)^2}{2 (1-\tau)^2} \\
& \ge \int_{p'}^p \frac{p'-r}{(1-\tau)^2 p} \bigg[ \binom{n}{2} \tau(1-\tau) - \mmse(\theta) \bigg] \, \rmd \theta - \frac{n^2 \tau (p-r)^2}{2 (1-\tau)^2} ,
\end{align*}
where the last inequality holds for any $p' \in [r,p]$ because $\mmse(\theta)$ is no larger than $\binom{n}{2} \tau (1-\tau)$, the trivial mean squared error achieved by the constant estimator $\E[X]$.
Finally, since $\mmse(\theta)$ is non-increasing in $\theta$ by Lemma~\ref{lem:mmse-monotone}, we obtain
\begin{align*}
\KL(\cP_A \| \cQ)
\ge \frac{(p-p')(p'-r)}{(1-\tau)^2 p} \bigg[ \binom{n}{2} \tau(1-\tau) - \mmse(p') \bigg] - \frac{n^2 \tau (p-r)^2}{2 (1-\tau)^2} .
\end{align*}
Rearranging the above inequality proves the first bound in the lemma.

The other bound is obtained in a similar way, using the second bound in Lemma~\ref{lem:mutual info upper bound 2}. We have $\Var_\theta(\Xtd_e) \le \frac{\theta^2}{q(\theta)^2} \Var_\theta(\Xhat_e) + \tau^2 \Big( \frac{\theta^2}{q(\theta)^2} - 1 \Big) \le \frac{(1-\tau)^2 \theta^3}{q(\theta)^2 (\theta-r)} \, \rone_e + \tau^2 \Big( \frac{\theta^2}{q(\theta)^2} - 1 \Big)$ where $q(\theta) := \frac{r - \tau \theta}{1-\tau}$.
Therefore, 
\begin{align*}
\mmse(\theta) 
&\geq \binom{n}{2} \tau(1-\tau) - n^2 \tau^2 \Big( \frac{\theta^2}{q(\theta)^2} - 1 \Big) - \frac{(1-\tau)^2 \theta^3}{q(\theta)^2 (\theta-r)} \sum_{e\in\binom{[n]}{2}} \rone_e . 
\end{align*}
Similar to the previous case, we have
\begin{align*}
\KL(\cP_A \| \cQ) 
& \ge \int_r^p \frac{q(\theta)^2 (\theta-r)}{(1-\tau)^2 \theta^3} \bigg[ \binom{n}{2} \tau(1-\tau) - \mmse(\theta) - n^2 \tau^2 \Big( \frac{\theta^2}{q(\theta)^2} - 1 \Big) \bigg] \, \rmd \theta \\
& \ge \int_{p'}^p \frac{q(\theta)^2 (p'-r)}{(1-\tau)^2 p^3} \bigg[ \binom{n}{2} \tau(1-\tau) - \mmse(\theta) \bigg] - \frac{n^2 \tau^2 (\theta^2 - q(\theta)^2) (\theta-r)}{(1-\tau)^2 \theta^3} \, \rmd \theta \\
& \ge \frac{q(p)^2 (p-p') (p'-r)}{(1-\tau)^2 p^3} \bigg[ \binom{n}{2} \tau(1-\tau) - \mmse(p') \bigg] - \frac{n^2 \tau^2 (p^2 - q(p)^2) (p-r) (p-p')}{(1-\tau)^2 (p')^3} .
\end{align*}
Rearranging the above inequality completes the proof.
\end{proof}

\subsection{Proof of Theorem~\ref{thm:recovery lower bound}}

Recall Theorem~\ref{thm:detection lower bound} which gives $\TV(\cP_A, \cQ) \to 0$.
This can strengthened as follows.


\begin{lemma} \label{lem:kl-go-to-zero}
Consider the setup of Theorem~\ref{thm:detection lower bound} where in particular $\tau \to 0$, $n \tau \ge (\log n)^3$, and $n\tau \lambda \to 0$ for $\lambda = \frac{(p-q)^2}{r(1-r)}$. 
In addition, assume that $r \ge 1/n$ and $n \lambda \tau^{3/2} \log n \to 0$ (which is implied by $n\tau \lambda \to 0$ if $\tau \le (\log n)^{-2}$).
Then we have $\KL(\cP_A \| \cQ) \to 0$.
\end{lemma}

\begin{proof}


Let $\Ecal = \cE_0 \cap \cE_1$ be the high-probability event defined in \eqref{eq:goodEvent}.
We first show that in fact $\cP(\cE^c) \le n^{-9}$ under the additional assumption $n \lambda \tau^{3/2} \log n \to 0$.
Recall Lemma~\ref{lem:bound_goal1} which shows that $\cP(\cE_0^c) \le c_n \to 0$. To have $c_n \le n^{-10}$, it suffices to choose $t = t_n \to 0$ in \eqref{eq:e-0-tail-bound} so that  
$$
\min \Big\{ \frac{t}{n \lambda \tau^{3/2}}, \frac{t^{2/3}}{n^{1/3} \lambda^{2/3} \tau}, \frac{t^{1/2}}{\lambda^{1/2} \tau^{1/2}} \Big\} \ge C \log n
$$
for a sufficiently large constant $C > 0$. This is possible 
thanks to the conditions $n \tau \ge (\log n)^3$, $n\tau \lambda \to 0$, and $n \lambda \tau^{3/2} \log n \to 0$.
Moreover, Lemma~\ref{lem:good event} already gives $\cP(\cE_1^c) \le n^{-10}$.
We conclude that $\cP(\cE^c) \le n^{-9}$.

Let $\cP(A, \cE) := \int_{\cE} \cP(A,z) \, \rmd z$. Then for any $\delta \in [0,1]$,
$$
\cP(A) 
= \cP(A, \cE) + \cP(A, \cE^c) 
= (1-\delta) \frac{\cP(A, \cE)}{1-\delta} + \delta \frac{\cP(A, \cE^c)}{\delta} .
$$
Since $x \mapsto x \log x$ is convex, Jensen's inequality implies that
\begin{align*}
\KL(\cP_A\| \cQ) & = \sum_{A} \Big( \cP(A) \log \cP(A) - \cP(A) \log \cQ(A) \Big) \\
& \le \sum_{A} \pr{ (1-\delta) \frac{\cP(A, \cE)}{1-\delta} \log \frac{\cP(A, \cE)}{1-\delta} + \delta \frac{\cP(A, \cE^c)}{\delta} \log \frac{\cP(A, \cE^c)}{\delta} - \cP(A) \log \cQ(A) } \\
& = \sum_{A} \pr{ \cP(A, \cE) \log \frac{\cP(A, \cE)}{(1-\delta) \cQ(A)} + \cP(A, \cE^c) \log \frac{\cP(A, \cE^c)}{\delta \cQ(A)} } .
\end{align*}
Set $\delta = 1/n$.
For the first term above, we apply the inequality $\log x \le x-1$; for the second term, we use the trivial bounds $\cP(A, \cE^c) \le 1$ and $\cQ(A) \ge r^{\binom{n}{2}} \ge n^{-n^2+1}$ since $r \ge 1/n$. Then we obtain
\begin{align*}
\KL(\cP_A\| \cQ) 
&\le \sum_{A} \pr{ \cP(A, \cE) \Big( \frac{\cP(A, \cE)}{(1-1/n) \cQ(A)} - 1 \Big) + \cP(A, \cE^c) \, n^2 \log n } \\
&= \frac{n}{n-1} \sum_{A} \frac{\cP(A, \cE)^2}{\cQ(A)} - \cP(\cE) + \cP(\cE^c) \, n^2 \log n .
\end{align*}
As we have seen that $\cP(\cE^c) \le n^{-9}$, it remains to show that 
$$
\sum_{A} \frac{\cP(A, \cE)^2}{\cQ(A)} = 1+o(1) .
$$
In fact, this is precisely \eqref{eq:goal-of-conditional-second-moment} and has been verified in the proof of Theorem~\ref{thm:detection lower bound} in Section~\ref{sec:proof-of-detect-lower}.
\end{proof}



We are ready to prove Theorem~\ref{thm:recovery lower bound} by combining Lemmas~\ref{lem:mmse-trivial-interpolation} and~\ref{lem:kl-go-to-zero}. 
Let $\pi := 2 p - r$.
We first note (a somewhat trivial fact) that $\pi \le 1$, so that it is valid to consider the model $\cP_\pi$ from Definition~\ref{def:interpolation model}.
If $p \le 1/2$, then obviously $\pi \le 1$.
Assume $p > 1/2$. Note that by the conditions $n \tau \ge (\log n)^3 \to \infty$ and $n \tau \lambda \to 0$, we have $\lambda = \frac{(p-q)^2}{r(1-r)} \to 0$ and so $p-q \to 0$. Since $q \le r \le 1/2 \le p$, we have $p,r \to 1/2$. Thus, $\pi = 2p - r \le 1$.

Moreover, we claim that
\begin{equation*}
\KL((\cP_\pi)_A \| \cQ) \to 0 .
\end{equation*}
Let 
$q(\pi) := \frac{r - \tau \pi}{1-\tau}$ and 
$\lambda(\pi) := \frac{(\pi-q(\pi))^2}{r(1-r)}$.
Since $\tau < 1/2$ and $r - q(\pi) = \frac{\tau}{1-\tau} (\pi - r) \le \pi - r$, we have $\pi - q(\pi) = \pi - r + r - q(\pi) \le 2(\pi - r) = 4(p-r) \le 4 (p-q(p))$. Therefore,
$\lambda(\pi) \le 16 \lambda(p)$ by definition.
The conditions $n \tau \lambda(p) \to 0$ and $n \lambda(p) \tau^{3/2} \log n \to 0$ thus imply $n \tau \lambda(\pi) \to 0$ and $n \lambda(\pi) \tau^{3/2} \log n \to 0$. 
By Lemma~\ref{lem:kl-go-to-zero}, we see that $\KL((\cP_\pi)_A \| \cQ) \to 0$.

Since the MMSE $\mmse(\theta)$ is non-increasing in $\theta$ by Lemma~\ref{lem:mmse-monotone} and $\lambda(\theta) = \frac{(\theta - q(\theta))^2}{r(1-r)}$ is increasing in $\theta$, we may assume without loss of generality that $n \tau \lambda(p) \to 0$ arbitrarily slowly, e.g., $\lambda(p) = \frac{(p-q)^2}{r(1-r)} \ge \frac{1}{n \tau \log n}$.
The rest of the proof is split into two cases.

\paragraph{Case 1:} $p \to 0$. Since $r = \tau p + (1-\tau) q$, we have $p - r = (1-\tau) (p-q)$.
We apply the first bound in Lemma~\ref{lem:mmse-trivial-interpolation} with $p$ replaced by $\pi$ and $p'$ replaced by $p$ (and thus $\cP$ replaced by $\cP_\pi$) to obtain
\begin{align*}
\binom{n}{2} \tau(1-\tau) - \mmse(p) 
&\le \Big( \KL((\cP_\pi)_A \| \cQ) + \frac{n^2 \tau (\pi - r)^2}{2 (1-\tau)^2} \Big) \frac{(1-\tau)^2 \pi}{(\pi-p)(p-r)} \\
&= \KL((\cP_\pi)_A \| \cQ) \cdot \frac{2p-r}{(p-q)^2} + 2 n^2 \tau (2p-r) .
\end{align*}
For the first term, we have $\frac{(2p-r)}{(p-q)^2} \le \frac{1}{(p-q)^2} = \frac{1}{\lambda r(1-r)} \le \frac{n \tau \log n}{r(1-r)} \le 2 n^2 \tau$ since $\frac{\log n}{n} \le r \le \frac 12$.
We have shown that $\KL((\cP_\pi)_A \| \cQ) \to 0$, so the first term is $o(n^2 \tau)$.
The second term is also $o(n^2 \tau)$ because $p \to 0$ by assumption.

\paragraph{Case 2:} $p \ge 2 \delta$ for a constant $\delta > 0$. As before, we have $\lambda = \frac{(p-q)^2}{r(1-r)} \to 0$ and so $p-q \to 0$. Since $q \le r \le p$, we also have $p-r \to 0$. It is then easily seen that $p,q,r,\pi,q(\pi) \ge \delta$.
We apply the second bound in Lemma~\ref{lem:mmse-trivial-interpolation} with $p$ replaced by $\pi$ and $p'$ replaced by $p$ to obtain
\begin{align*}
\binom{n}{2} \tau(1-\tau) - \mmse(p) 
&\le \Big( \KL((\cP_\pi)_A \| \cQ) + \frac{n^2 \tau^2 (\pi^2 - q(\pi)^2) (\pi-r) (\pi-p)}{(1-\tau)^2 p^3} \Big) \frac{(1-\tau)^2 \pi^3}{q(\pi)^2 (\pi-p) (p-r)} \\
&= \KL((\cP_\pi)_A \| \cQ) \cdot \frac{(1-\tau)^2 \pi^3}{q(\pi)^2 (p-r)^2} + \frac{2 n^2 \tau^2 \pi^3 (\pi^2 - q(\pi)^2)}{p^3 q(\pi)^2} \\
&\le \KL((\cP_\pi)_A \| \cQ) \cdot \frac{1}{\delta^2 (p-q)^2} + \frac{2 n^2 \tau^2}{\delta^5} .
\end{align*}
Similar to the previous case, $\frac{1}{(p-q)^2} \le 2n^2 \tau$ and so the first term is $o(n^2 \tau)$. 
The second term is also $o(n^2 \tau)$ because $\tau \to 0$ by assumption.

\section{Upper bounds}
\label{sec:upper-bounds}


For both the detection problem from Definition~\ref{def:prob-detect} and the recovery problem from Definition~\ref{def:prob-recover}, the upper bounds are achieved via the same optimization program defined as follows.
In the sequel, for $X, A \in \R^{n \times n}$, we use the notation
$$
\abs{X} := \sum_{1 \le i < j \le n} |X_{ij}|, \qquad \langle X, A \rangle := \sum_{1 \le i < j \le n} X_{ij} A_{ij} .
$$
Define
\begin{equation}
\label{eq:feasible_set-no-constraint}
\Xfrak := \left\{X \in \{0,1\}^{\binom{n}{2}} : \text{there is } z \in [0,1]^n \text{ such that } X_{ij} = \bbone_{\dist(z_i,z_j) \le \tau/2} \text{ for all } (i,j) \in \binom{[n]}{2} \right\} ,
\end{equation}
i.e., $\Xfrak$ is the set of all noiseless geometric graphs that can be realized by $z$ as in Definition~\ref{def:mod-p}.
Moreover, let
\begin{equation}
\label{eq:feasible_set}
\tilde \Xfrak := \br{X \in \Xfrak : \abs{|X| - \binom{n}{2}\tau} \le n \tau^{1/2} \log n } .
\end{equation}
Then we consider the maximization problem
\begin{equation}\label{eq:estimator_def}
\hat X := \argmax_{X' \in \tilde \Xfrak} \, \langle X', A \rangle , \qquad
\hat L := \langle \hat X, A \rangle .
\end{equation}

The following result for detection holds.


\begin{theorem}
\label{thm:upper bound on detection}
Consider the detection problem from Definition~\ref{def:prob-detect}.
Suppose that 
$\frac{n \tau (p-r)^2}{r} \ge C \log n$ and $n \tau (p - r) \ge C \log n$ for a sufficiently large constant $C > 0$. 
Let $\hat L$ be defined in \eqref{eq:estimator_def} and $\kappa := \binom{n}{2}\tau \frac{p+r}{2}$. 
Then the test $\1_{\hat L > \kappa}$ achieves strong detection in the sense of Definition~\ref{def:prob-detect}.
\end{theorem}

\begin{proof}
First, suppose that $(A, X) \sim \cP$ as in Definition~\ref{def:mod-p}. 
By the proof of Lemma~\ref{lem:bound_goal1}, 
$$
\p \left\{ \abs{\abs{X} - \binom{n}{2}\tau} > t \right\} 
\le K \exp\br{-\frac{1}{K}\min\kr{\frac{t}{n \tau^{1/2}}, \frac{t^{2/3}}{n^{1/3} \tau^{1/3}}, t^{1/2}}} 
$$
for a constant $K>0$.
Since $n^2 \tau \to \infty$, taking $t = n \tau^{1/2} \log n$ gives that
\begin{equation}
\abs{\abs{X} - \binom{n}{2}\tau} \le n \tau^{1/2} \log n
\label{eq:x-abs-size-bd}
\end{equation}
with probability $1-o(1)$.
We condition on a realization of $X$ such that the above bound holds.
In particular, $X \in \tilde \Xfrak$. 
By \eqref{eq:estimator_def}, we have $\hat L \geq \langle X, A \rangle$ where the lower bound has distribution $\Bin(|X|, p)$ under $\cP$.
Since $n^2 \tau p \to \infty$ (implied by the assumption $n \tau (p - r) \ge C \log n$), Bernstein's inequality together with \eqref{eq:x-abs-size-bd} yields that
\begin{equation*}
\hat L 
\geq \binom{n}{2} \tau p - p n \tau^{1/2} \log n - n \sqrt{\tau p} \log n
\geq \binom{n}{2} \tau p - 2 n \sqrt{\tau p} \log n
\end{equation*}
with probability at least $1-o(1)$. 
Putting it together, for $\hat L > \kappa = \binom{n}{2}\tau \frac{p+r}{2}$ to hold, we need 
$\binom{n}{2} \tau \frac{p - r}{2} > 2 n \sqrt{\tau p} \log n,$
which 
is implied by the assumption $n \tau (p - r) \ge C \log n$.

Next, suppose that $A \sim \cQ$.
Fix any $X' \in \tilde\Xfrak$. 
Then $\langle X', A \rangle \sim \Bin(\abs{X'},r)$.
Bernstein's inequality then gives that, with probability at least $1 - n^{-9 n}$, 
$$
\langle X', A \rangle 
\le |X'| r + C_1 \Big( \sqrt{|X'| r n \log n} + n \log n \Big) 
$$
for a constant $C_1 > 0$.
Then, by \eqref{eq:x-abs-size-bd}, Lemma~\ref{lem:deterministic upper bound on number of candidates} and a union bound over $X' \in \tilde \Xfrak$, we have that with probability $1-o(1)$,
\begin{equation*}
\hat L \le \binom{n}{2} \tau r 
+ C_2 \Big( \sqrt{n^3 \tau r \log n} + n \log n \Big) 
\end{equation*}
for a constant $C_2 > 0$.
For $\hat L < \kappa = \binom{n}{2}\tau \frac{p+r}{2}$ to hold, we need 
$\binom{n}{2} \tau \frac{p - r}{2} > C_2 ( \sqrt{n^3 \tau r \log n} + n \log n ).$
It suffices to have
$\frac{n \tau (p-r)^2}{r} \ge C_3 \log n$ and $n \tau (p - r) \ge C_3 \log n$ for a constant $C_3 > 0$.
\end{proof}

\begin{lemma}
\label{lem:deterministic upper bound on number of candidates}
We have
$|\Xfrak| \leq n^{3n}$ for $\Xfrak$ defined in \eqref{eq:feasible_set-no-constraint}.
\end{lemma}

\begin{proof}
Consider $n$ points $z_1, \dots, z_n$ on the unit circle.
Recall that $X_{ij} = \bbone_{\dist(z_i,z_j) \le \tau/2}$ for $(i,j) \in \binom{[n]}{2}$. 
For simplicity, we start with an arbitrary point and relabel the $n$ points as $z_1, \dots, z_n$ clockwise on the circle; this relabeling incurs at most an $n!$ factor when counting $X \in \Xfrak$.
Since $\tau < 1/2$, the neighbors of $z_1$ in the graph $X$ must be of the form $z_2, z_3, \dots, z_i$ on one side and $z_j, z_{j+1}, \dots, z_n$ on the other side.
Thus, there are at most $n^2$ possible choices for the neighborhood of vertex $1$ in the graph $X$.
The same holds for any of the $n$ vertices.
In conclusion, $|\Xfrak| \le n! (n^2)^n \le n^{3n}$.
%
\end{proof}

Next, we turn to the recovery problem.

\begin{theorem}
\label{thm:recovery-upper-bound-thm}
Consider the recovery problem from Definition~\ref{def:prob-recover}.
Suppose that $\frac{n \tau (p-q)^2}{p \log n} \to \infty$ and $\frac{n \tau (p-q)}{\log n} \to \infty$. 
Then the estimator $\hat X$ defined in \eqref{eq:estimator_def} achieves strong recovery in the sense of Definition~\ref{def:prob-recover}.
\end{theorem}



\begin{proof}
As shown in the proof of Theorem~\ref{thm:upper bound on detection}, we have $X \in \tilde \Xfrak$ with probability $1-o(1)$.
Condition on such an instance of $X$ in the sequel.
Fix $X' \in \tilde \Xfrak$ such that $|X'-X| \ge \eps n^2 \tau$ for a fixed $\eps > 0$, where we use the notation $|X'-X| := \sum_{1 \le i < j \le n} |X'_{ij} - X_{ij}|$.
Let $X \setminus X' \in \{0,1\}^{\binom{n}{2}}$ be defined by $(X \setminus X')_{ij} = \bbone_{X_{ij} = 1, \, X'_{ij} = 0}$.
Then
\begin{equation*}
\langle X, A \rangle - \langle X', A \rangle
= \langle X \setminus X', A \rangle - \langle X' \setminus X, A \rangle ,
\end{equation*}
where $\langle X \setminus X', A \rangle$ and $\langle X' \setminus X, A \rangle$
are independent $\Bin(\abs{X \setminus X'}, p)$ and $\Bin(\abs{X' \setminus X}, q)$ random variables respectively.
Therefore,
\begin{align*}
\E[ \langle X \setminus X', A \rangle ] - \E[ \langle X' \setminus X, A \rangle ]
&= \abs{X \setminus X'} \, p - \abs{X' \setminus X} \, q \\
& = \frac{p-q}{2} |X'-X| + \frac{p+q}{2} \pr{\abs{X} - \abs{X'}}     \\
&\ge \frac{p-q}{2} \eps n^2 \tau - (p+q) n \tau^{1/2} \log n 
\end{align*}
as $X, X' \in \tilde \Xfrak$.
It then follows from Bernstein's inequality that with probability at least $1 - n^{-9n}$, 
$$
\langle X, A \rangle - \langle X', A \rangle
\ge \frac{p-q}{2} \eps n^2 \tau 
- C_1 \Big( \sqrt{\eps n^3 \tau p \log n} + n \log n \Big) 
$$
for a constant $C_1 > 0$.
By Lemma~\ref{lem:deterministic upper bound on number of candidates} and a union bound, the above inequality holds for all such $X'$ with probability at least $1-n^{-6n}$.
We now claim that the right-hand side of the above bound is positive.
It suffices to have
$\frac{\eps n (p-q)^2 \tau}{p \log n} \ge C_2$ and $\frac{\eps n (p-q) \tau}{\log n} \ge C_2$ for a constant $C_2 > 0$. 
Since $\frac{(p-q)^2 n \tau}{p \log n} \to \infty$ and $\frac{(p-q) n \tau}{\log n} \to \infty$, the claim is proved.

We have shown that, with probability at least $1-n^{-6n}$, $\langle X, A \rangle - \langle X', A \rangle > 0$ for all $X' \in \tilde \Xfrak$ such that $|X'-X| \ge \eps n^2 \tau$. By the definition of $\hat X$ in \eqref{eq:estimator_def}, it follows that $|\hat X - X| \le \eps n^2 \tau$.
Since $|X|, |\hat X| \le n^2$, we conclude that $\E[|\hat X - X|] \le 2 \eps n^2 \tau$.
Note that $R(\hat X, X)$ in Definition~\ref{def:prob-recover} is precisely $\E[|\hat X - X|]$ since $X_{ij}, \hat X_{ij} \in \{0,1\}$. As $\eps>0$ is arbitrary, the proof is complete. 
\end{proof}

\appendix

\section{Existing results on U-statistics}
\label{sec:u-stats}



We state a few existing results on $U$-statistics, starting with decoupling inequalities \cite{de1995decoupling,delapenaDecoupling1999,houdreExponentialInequalitiesConstants2003}. 
We only use $U$-statistics of order 2, in which case the general results specialize to the following lemma.

\begin{lemma}
[Theorems~3.1.1 and~3.4.1 in \cite{delapenaDecoupling1999}]
\label{lem:decoupling}
Let $z_1, \dots, z_n$ be $n$ independent random variables in $[0,1]$ and write $z = (z_1, \dots, z_n)$. Let $z''$ be an independent copy of $ z$. 
For any $(i,j) \in \binom{[n]}{2}$, let $f_{ij} : [0,1]^2 \to \mathbb R$ be a function such that $\E|f_{ij}(z_i,z_j)| < \infty$.
\begin{itemize}
\item
For any convex nondecreasing function $\Psi:[0,\infty)\to [0,\infty)$ such that $\E \Psi(|f_{ij}(z_i,z_j)|) < \infty$ for all $(i,j) \in \binom{[n]}{2}$, we have
\begin{equation*}
\E \Psi \pr{\abs{\sum_{1\leq i<j\leq n} f_{ij}(z_i,z_j)}} \leq \E \Psi \pr{12 \abs{\sum_{1\leq i<j \leq n} f_{ij}(z_i,z_j'')}} .
\end{equation*}

\item
There exists an absolute constant $C>0$ such that for all $t > 0$,
\begin{equation*}
\Pb\pr{\abs{\sum_{1\leq i<j\leq n} f_{ij}(z_i,z_j)} > t} \leq C \, \Pb\pr{C\abs{\sum_{1\leq i<j\leq n} f_{ij}(z_i,z_j'')} > t}.
\end{equation*}
\end{itemize}
\end{lemma}

We also need moment inequalities and tail bounds for decoupled $U$-statistics \cite{gineExponentialMomentInequalities2000}.

\begin{lemma}
[Equations (3.3) and (3.5) in \cite{gineExponentialMomentInequalities2000}] \label{lem:hti moments}
Let $z_1, \dots, z_n$ be $n$ independent random variables in $[0,1]$ and write $z = (z_1, \dots, z_n)$. Let $z''$ be an independent copy of $ z$.
For any $(i,j) \in \binom{[n]}{2}$, let $f_{ij} : [0,1]^2 \to \mathbb R$ be a function such that $\E[f_{ij}(z_i,z''_j)]=0$.
For brevity, write $f_{ij} = f_{ij}(z_i,z_j'')$. Then there exists an absolute constant $K>0$ such that the following holds:
\begin{itemize}
\item 
For any $d \ge 2$, 
\begin{align*}
\E \abs{\sum_{1\leq i<j\leq n} f_{ij}}^d \leq  K^d \max\Bigg\{ & d^d\pr{\sum_{1\leq i<j\leq n} \E  f_{ij}^2 }^{d/2}, d^{3d/2}\E_{\zb} \max_{1\leq i< n}\pr{\sum_{j = i+1}^{n} \E_{\zb''}  f_{ij}^2}^{d/2}, \\
& d^{3d/2}\E_{\zb''} \max_{1< j\leq n}\pr{\sum_{i=1}^{j-1} \E_{\zb}  f_{ij}^2}^{d/2}, d^{2d}\E \max_{1\leq i<j\leq n}\abs{ f_{ij}}^d \Bigg\} .
\end{align*}


\item
For any $t > 0$,
\begin{equation*}
\Pb\br{\abs{\sum_{1\leq i<j\leq n}  f_{ij}} > t } \leq \constTailbound \exp\br{-\frac{1}{\constTailbound}\min\kr{\frac{t}{C},\pr{\frac{t}{B}}^{2/3},\pr{\frac{t}{A}}^{1/2}}} ,
\end{equation*}
where
\begin{align*}
A &= \max_{1\leq i<j\leq n} \sup_{x,y \in [0,1]} |f_{ij}(x,y)|, \quad C^2 = \sum_{1\leq i<j\leq n} \E  f_{ij}^2 , \\
B^2 &= \max \Bigg\{ \max_{1\leq i < n}\sup_{x \in [0,1]} \pr{\sum_{j=i+1}^n \E_{z_j''} f_{ij}(x,z_j'')^2} ,  \max_{1< j \leq n}\sup_{y \in [0,1]}\pr{\sum_{i=1}^{j-1} \E_{z_i}  f_{ij}(z_i,y)^2} \Bigg\}.
\end{align*}
\end{itemize}
\end{lemma}

\section*{Acknowledgments}
We thank Vladimir Koltchinskii for a discussion on U-statistics, thank Sam van der Poel for a discussion on one-dimensional geometric graphs, and thank Jiaming Xu for a discussion on the I-MMSE method.

\bibliographystyle{alpha}
\bibliography{MyLibrary,IT}

\end{document}